\documentclass[a4paper,11pt]{article}
\usepackage[utf8x]{inputenc}

\usepackage{fullpage}
\usepackage{amsmath,amsfonts,amssymb,amsthm}

\usepackage{color}
\usepackage{mathtools}

\renewcommand{\d}{{\sqrt{d}}}
\providecommand{\R}{\mathbb{R}}
\providecommand{\Z}{\mathbb{Z}}
\providecommand{\ZL}{\mathbb{Z}^d\cap\mathbb{T}_L}
\providecommand{\E}{\mathbb{E}^d_L}
\providecommand{\V}{\frac{1}{2\rho}}
\providecommand{\ud}[1]{\, \mathrm{d} #1}
\providecommand{\dx}{\ud{x}}
\providecommand{\dy}{\ud{y}}
\providecommand{\dz}{\ud{z}}
\providecommand{\oh}{\frac{1}{2}}
\providecommand{\ep}{\varepsilon}
\DeclareMathOperator{\osc}{osc} 
\newcommand{\en}[1]{\left< #1 \right>}
\newcommand{\p}[1]{\left(#1\right)}
\renewcommand{\a}[1]{\left| #1 \right|}
\renewcommand{\aa}[1]{\left\| #1 \right\|}
\renewcommand{\div}{\nabla \cdot} 
\providecommand{\grad}{\nabla}
\providecommand{\sym}{\mathrm{sym}\,}
\providecommand{\T}{\mathbb{T}_L}
\renewcommand{\gg}{\grad_x \grad_y}

\newtheorem*{definition*}{Definition}
\newtheorem{theorem}{Theorem}

\newtheorem{lemma}{Lemma}[section]
\newtheorem*{lemma*}{Lemma}
\newtheorem{remark}[lemma]{Remark}
\newtheorem{cor}[lemma]{Corollary}
\newtheorem*{example}{Example}

\allowdisplaybreaks 

\title{Corrector estimates for elliptic systems with random periodic coefficients}
\author{Peter Bella\footnote{Max Planck Institute for Mathematics in the Sciences, Leipzig (Germany), email: bella@mis.mpg.de} \and Felix Otto\footnote{Max Planck Institute for Mathematics in the Sciences, Leipzig (Germany), email: otto@mis.mpg.de}
}

\begin{document}

\maketitle

\begin{abstract}
 We consider an elliptic system of equations on the torus $\left[ -\frac{L}{2}, \frac{L}{2} \right)^d$ with random coefficients $A$, that are assumed to be coercive and stationary. Using two different approaches we obtain moment bounds on the gradient of the corrector, independent of the domain size $L$. In the first approach we use Green function representation. For that we require $A$ to be locally H\"older continuous and distribution of $A$ to satisfy Logarithmic Sobolev inequality. The second method works for non-smooth (possibly discontinuous) coefficients, and it requires that statistics of $A$ satisfies Spectral Gap estimate. 
\end{abstract}



\section{Introduction}

We are interested in the homogenization of linear second order elliptic system of equations with {\it random coefficients} of the form
\begin{equation}
 - \div \p{ A\p{\frac{x}{\ep}} \nabla u_\ep(x) } = f(x).
\end{equation}
It is a well-known fact that if distribution of $A$ is stationary and ergodic, then as $\ep \to 0$ the solution $u_\ep$ converges a.s. to $u_0$ -- a solution of an elliptic system with deterministic and constant {\it homogenized} coefficient field $A_{hom}$. This was proved by Kozlov \cite{kozlov} and independently by Papanicolaou and Varadhan \cite{papvar}.\footnote{Both of these articles state their results only for equations, but their methods extends also to the systems case.} For obvious practical reasons it is important not only to show the convergence of $u_{\ep}$, but also to be able to compute $A_{hom}$. If the coefficient field $A$ is periodic and deterministic, given direction $e$ ($e$ is a vector or a matrix in the case of one equation or a system of equations, respectively), this can be done by considering the notion of a corrector $\phi$, the unique periodic solution with zero mean to the cell problem $-\div (A(y) (e + \nabla \phi(y) ) = 0$, and defining $A_{hom}$ using an expression for the energy density $e A_{hom} e = \int (e + \nabla \phi(y)) A(y) (e + \nabla \phi(y))$.\footnote{Here we assumed that $A$ is elliptic in some sense (see discussion on different notions of ellipticity below). If this was not the case, to obtain the formula for $A_{hom}$ one would need to consider the equation for the corrector over multiple of cells (see work of M\"uller \cite{muller87} for a similar result in a more general setting).} 

If $A$ is random, the corrector satisfies the equation in the whole space, and a similar formula for $A_{hom}$ holds (with the integral on the right-hand side replaced by the average over the probability space). To compute the corrector for system with random coefficients one needs to solve the equation in the whole space, which is numerically very difficult. Moreover, in some cases (for example in dimension 2) even the notion and existence of a stationary corrector is not clear. 

To remove these two possible caveats, as a proxy for the original problem we consider the case of random but periodic $A$. In this case the corrector is defined as a periodic solution on the torus with zero mean, it exists and is unique, and it is less difficult to compute it numerically (at least for few realizations of a coefficient field $A$). If $A$ is $L$-periodic, it is natural to expect that $e A_{hom} e$ can be well approximated by 
\begin{equation}\label{2}
 \frac{1}{L^d} \int_{[0,L)^d} \p{ \nabla \phi(A;x) + e} A(x) \p{ \nabla \phi(A;x) + e} \dx,
\end{equation}
where the corrector $\phi$ is an $L$-periodic solution to $-\div A (\nabla \phi(A;\cdot) + e) = 0$.\footnote{Matrix $e$ will be fixed throughout the whole paper, and so we will not explicitly write that the corrector $\phi$ depends on $e$.} Moreover, to improve the error coming from approximating $eA_{hom}e$ by \eqref{2}, we can average \eqref{2} over several realizations of $A$. In order to quantify this error, one needs to estimate variance of \eqref{2}. 
In this paper we will present two quite different techniques how to obtain such estimate. 

In contrast with the qualitative theory of homogenization of equations with random coefficients \cite{kozlov,papvar} (see also \cite{kozlov2,kunnemann} for similar results for discrete elliptic equations), where stationarity and ergodicity of statistics of coefficient fields is enough to guarantee homogenization, the quantitative theory requires stronger, quantitative version of ergodicity. Quantifying ergodicity in the form of uniform mixing condition (i.e., assuming algebraic decay of correlations), Yurinski\u\i\ \cite{yurinski86} was the first to prove the rate of convergence (though not optimal) of a solution to an elliptic equation with random coefficients to the solution of a homogenized equation. Later, together with Pozhidaev, Yurinski\u\i\ extended this result to systems of equations \cite{pozyur89}. Assuming small ellipticity contrast ratio (requirement for the Meyers estimate to hold for exponents $p = 4$), in the case of a discrete elliptic equation with diagonal coefficients, Naddaf and Spencer showed in their unpublished work \cite{naddafspencer98} the optimal rate of convergence. To our knowledge, in this setting they were the first to quantify the ergodicity of the space of coefficient fields using the {\it Spectral Gap inequality} (SG), which they derived from the Brascamb--Lieb inequality. Inspired by the work of Naddaf and Spencer, Gloria and Otto \cite{GO1} improved the result of \cite{naddafspencer98} and obtained the optimal estimates for the random error without the assumption of small ellipticity contrast. In companion article \cite{GO2}, Gloria and Otto obtained the optimal estimate also for the systematic error (see \cite{GNO2} or \cite{GO2} for the definition of the random and the systematic error). Together with Neukamm, Gloria and Otto \cite{GNO1} estimated the error between the random solution and the first two terms of the asymptotic expansion. In \cite{GNO1}, instead of (SG) they assumed 
the {\it Logarithmic Sobolev inequality} (LSI) (which is a little stronger than (SG)) and use Green function estimates obtained by Marahrens and Otto \cite{MO}. 

Though most of the previous results were proved for discrete elliptic equations, we believe it should be possible to use similar methods to extend some of these results also to the case of linear elliptic equations in $\mathbb{R}^d$. On the other hand, since most of the previous arguments are based on the regularity theory for scalar elliptic operators, connected with names of De Giorgi, Nash, and Moser (often based on maximum principle, which is not available for systems of equations or discrete equations with non-diagonal coefficients), to treat the case of systems of elliptic equations one needs to use different methods. In a recent work, Ben Artzi, Marahrens, and Neukamm \cite{AMN} obtained estimates on the gradient and second mixed gradient of the Green function for discrete elliptic equation with non-diagonal coefficient. These estimates, used together with  Logarithmic Sobolev inequality and the Spectral Gap inequality, allow them to get estimates on the gradient of the corrector and the corrector itself, respectively. Since in their setting the maximum principle does not hold, there is a hope their methods could be used to study homogenization of a system of equations in $\Z^d$ as well. 

In the nonlinear setting, the only known quantitative result for homogenization of convex integral functionals is the recent work of Armstrong and Smart \cite{armstrongsmart2014}, who extended the qualitative result due to Dal Maso and Modica \cite{dalmasomodica}. Armstrong and Smart used clever cut and paste technique, which for any two open sets $U,V \subset \R^d$, separated by distance at least $1$, requires the statistics of coefficient fields $A$ in $U$ to be independent of the statistics of $A$ in $V$. This assumption replaces the Spectral Gap inequality (or the Logarithmic Sobolev inequality) used in the previously mentioned articles. As a special case (using their result for quadratic functionals), they prove homogenization also for linear elliptic equations. In contrast with our approach they used variational techniques, which in the linear setting would require coefficients $A$ to be symmetric, while we do not need this assumption. The result in \cite{armstrongsmart2014} is stated for scalar functionals, but using their methods it should be possible to extend the result also to the case of nonlinear convex vectorial functionals. 

We assume the statistics of coefficient fields is stationary, meaning that $A$ and $A(z+\cdot)$ have the same distribution. In our and also in many already mentioned works there are basically two main assumptions besides stationarity: the first one is deterministic, and assumes that coefficient fields $A$ are in some sense elliptic; the latter is probabilistic, and asserts that the distribution of coefficient fields is ergodic in a quantitative way. 

In contrast with scalar equations, where there are not many different notions of ellipticity, for systems there are several possible choices. A stronger condition, called {\it very strong ellipticity} (also known as the Legendre condition), assumes $M\cdot A(x)M \ge \lambda \a{M}^2$ and $\a{A(x)M} \le \a{M}$, uniformly in $x$ and for all matrices $M$. A weaker notion of ellipticity is the one of {\it strong ellipticity} (also known as the Legendre-Hadamard condition), where the first inequality is assumed only for {\it rank-1} matrices $M$. In both of these, $\lambda \in (0,1)$ is fixed. In the case of the whole space (or a torus), it is obvious that very strongly elliptic $A$ satisfies
\begin{equation}\label{discA}
 \int \nabla \varphi(x) \cdot A(x) \nabla \varphi(x) \dx \ge \lambda \int \a{ \nabla \varphi(x) }^2 \dx
\end{equation}
for any $\varphi \in W^{1,2}$. If $A$ is only strongly elliptic but constant (i.e., it does not depend on $x$), \eqref{discA} still holds. This can be seen from Plancherel Theorem and the fact that in the Fourier space $\nabla \phi$ is a rank-1 matrix. For general non-constant strongly elliptic $A$ \eqref{discA} fails to hold.\footnote{For a uniformly continuous strongly elliptic $A$, Garding's inequality implies weaker version of \eqref{discA} with added multiple of $\int \varphi(x)^2 \dx$ on the left-hand side.} On the other hand, it can be proved that if $A$ satisfies \eqref{discA}, then for a.e. $x$ (for every $x$ if $A$ is continuous) the matrix $A(x)$ is strongly elliptic. Hence we see that \eqref{discA} lies between the strong ellipticity and the very strong ellipticity. In both of our approaches we assume that all coefficient fields satisfy \eqref{discA}. 

The randomness of coefficient fields $A \in \Omega$ will be modelled by a probability measure on $\Omega$. Following convention in statistical mechanics, we call this probability measure an {\it ensemble} and denote by $\en{ \cdot }$ the expectation with respect to this measure (the {\it ensemble average}). To quantify ergodicity of the ensemble we assume it satisfies either Logarithmic Sobolev inequality (LSI) (see Theorem \ref{thm1}) or Spectral Gap inequality (SG) (see Theorem \ref{thm2}). Since there are several versions of these inequalities, let us quickly discuss few of them. In the discrete setting $\Z^d$, for simplicity in the case $\Omega = \left\{ a = \textrm{diag}(a_1,\ldots,a_d) \in \R^{d\times d} : \lambda \le a_i \le 1 \right\}^{\Z^d}$, one possible form of (SG) is the following: there exists $\rho$ such that for any random variable $\zeta \in L^2(\Omega)$:
\begin{equation}\label{SGv1}
 \en{ \p{ \zeta - \en{\zeta} }^2 } \le \frac{1}{\rho} \sum_{y \in \Z^d} \en{ \p{ \frac{\partial \zeta}{\partial a(y)} }^2 }.
\end{equation}
In this form, (SG) was considered by Naddaf and Spencer in their unpublished work \cite{naddafspencer98}, and can be seen as a Poincar\'e estimate in the infinite dimensional setting. By replacing $\p{ \frac{\partial \zeta}{\partial a(y)} }^2$ on the right-hand side with different terms, Gloria and Otto \cite{GO1,GO2}, Gloria, Otto, and Neukamm \cite{GNO2}, and Ben-Artzi, Marahrens, and Neukamm \cite{AMN} considered several, weaker notions of \eqref{SGv1}. In particular, in \cite[equation (10)]{AMN}, this term is replaced by $\p{ \osc_{a(y)} \zeta }^2$, where $\osc_{a(y)}$ stands for the oscillation w.r.t.  $a(y)$:
\begin{multline}\label{osc}
 \osc_{a(y)} \zeta := \sup\left\{ \zeta(\tilde a)\ |\ \tilde a \in \Omega \textrm{ s.t. } \tilde a(x) = a(x)\ \forall x\neq y \right\}
\\ - \inf\left\{ \zeta(\tilde a)\ |\ \tilde a \in \Omega \textrm{ s.t. } \tilde a(x) = a(x) \ \forall x\neq y \right\}.
\end{multline}
Compared to \eqref{SGv1}, (SG) with oscillation on the right-hand side holds for more general ensembles -- for example for i.i.d. associated with a single-site distribution that only assumes finite number of values (Bernoulli). 

A little stronger notion than (SG) is the Logarithmic Sobolev inequality (LSI) (see, e.g., \cite[Theorem 4.9]{notesLSI}, for the proof that (LSI) implies (SG)). In the discrete setting that would mean considering \eqref{SGv1} with the left-hand side $\en{ \p{ \zeta - \en{\zeta} }^2 }$ replaced by $\en{ \zeta^2 \log \frac{\zeta^2}{\en{\zeta^2}} }$. As was the case for (SG), in (LSI) one could also use different versions of the right-hand side. In our first result we assume the ensemble satisfies (LSI) with continuum derivatives on the right-hand side, while in the second result we will assume (SG) with oscillations on the right-hand side (see Theorem \ref{thm1} and Theorem \ref{thm2} for the precise form of these assumptions).

In the next section we define all the relevant notions, state our main results (Theorem \ref{thm1} and Theorem \ref{thm2}), and quickly discuss their proofs. In Section \ref{sec:disc} we state and prove Theorem \ref{thm3} -- a discrete version (both in terms of the statement and the idea of the proof) of Theorem \ref{thm1}. In Section \ref{sec:pf1} we present the main ingredients in the proof of Theorem \ref{thm1}, and afterwards (Section \ref{sec:pf1+}) we give arguments for those. Finally, the proof of Theorem \ref{thm2} is given in Section \ref{sec:pf2}. 

{\bf Notation.} We denote by $d \ge 2$ the dimension of the underlying space $\R^d$, by $n \ge 1$ number of equations, and by $L > 0$ the side-length of the $d$-dimensional torus $\T = \R^d / L\Z^d$. Given $r>0$ and $z \in \T$, $B_r(z)$ denotes a ball in $\T$, centered at $z$ with radius $r$; $B_r$ will stand for a ball $B_r(0)$. Here and throughout the paper, balls like $B_r(z)$ refer to the distance function on the torus. 

Though we consider a system of equations, it would be convenient to use scalar notation. For that we consider $Y$, a real Hilbert space of $\dim Y = n$ (sometimes we will identify $Y$ with $\R^n$). We denote by $zy$ and $z\cdot y$ respectively the inner product in $Y$ and the natural one induced over $Y^d$. In the same spirit we write $\a{z} = \p{zz}^{\frac{1}{2}}$ for $z \in Y$ and $\a{y} = \p{ y \cdot y }^{\frac{1}{2}}$ for $y \in Y^d$. For $y \in Y, u \in \R^d$, we denote by $y \otimes u \in Y^d$ the usual tensor product.

\section{Setting and the main results}\label{sec:def}

We start by introducing the relevant deterministic notions: The corrector $\phi(A;\cdot)$ and the homogenized coefficient $A_{hom}(A)$ for an arbitrary coefficient field $A$ on the torus $\T$.

\begin{definition*}[\sc Space of coefficient fields] 
Let $\Omega$ be the space of all $L^d$-periodic fields $A : \T \to \mathcal{L}(Y^d,Y^d)$ that are elliptic in the following sense: there exists $0 < \lambda < 1$, which is fixed throughout the paper, such that for any $A \in \Omega$ 
 \begin{equation}\label{ellipticity}
  \begin{aligned}
  \forall \varphi \in W^{1,2}(\T,Y):& \quad \int_{\T} \nabla \varphi(x) \cdot A(x) \nabla \varphi(x) \dx \ge \lambda \int_{\T} \a{ \nabla \varphi(x) }^2 \dx, 
   \\ \forall x \in \T, y \in Y^d:& \quad \a{ A(x) y } \le \a{y}.
  \end{aligned}
 \end{equation}
\end{definition*}

We point out that we do not assume {\it symmetry} of coefficient fields $A$. Since the first condition in \eqref{ellipticity} is not pointwise, let us mention few sufficient pointwise conditions on $A$ for \eqref{ellipticity} to hold. Recall that it is enough to assume that $A$ is very strongly elliptic in the sense that $y\cdot A(x)y \ge \lambda \a{y}^2$ for all $y \in Y^d$. Motivated by linear elasticity, in the case $d=n$ it follows from Korn's inequality that it is enough to assume that $A$ is {\it Korn-elliptic} in the sense that $y\cdot A(x)y \ge \lambda \a{\sym y}^2$ for all $y \in Y^d$. Here, we identified $Y^d$ with the space of matrices $\R^{d\times d}$, and by $\sym$ we denoted the symmetric part of a matrix. Finally, Pozhidaev and Yurinski\u\i\ \cite{pozyur89} gave a condition which generalizes both of these conditions: they assume that for each $A \in \Omega$ there exists $\bar A \in \mathcal{L}(Y^d,Y^d)$ such that $y \cdot A(x) y \ge y \cdot \bar A y$ for all $y \in Y^d, x \in \T$, and $\bar A$ is strictly rank-1 elliptic in the sense that $(y\otimes u) \cdot \bar A (y \otimes u) \ge \lambda \a{y}^2 \a{u}^2$ for all $y \in Y, u \in \R^d$. 

Let us now define the notion of a corrector:

\begin{definition*}[\sc Corrector] For given coefficient field $A\in\Omega$, the corrector $\phi(A;\cdot) : \T \to Y$ is the unique solution of 
\begin{equation}\label{T.3}
-\nabla\cdot A(\nabla\phi(A;\cdot)+e)=0\; \textrm{ in } \T \quad \mbox{and}\quad\int_{\T}\phi(A;x) \dx=0,
\end{equation}
where $e \in Y^d$ with $|e|=1$ is a ``direction'' that is fixed throughout the paper. We note that the uniqueness of $\phi$ implies ``stationarity'' in the sense of
\begin{equation}\label{L2.11}
\phi(A(\cdot+z),x)=\phi(A,x+z).
\end{equation}
\end{definition*}

\begin{definition*}[\sc Homogenized coefficient] Given $A \in \Omega$, the homogenized coefficient in direction $e$ is defined via
\begin{equation}\label{T.2}
A_{hom}(A)e:=L^{-d}\int_{\T}
A(x)(\nabla\phi(A;x)+e) \dx.
\end{equation}
\end{definition*}

We will assume that the probability measure (the {\it ensemble}) on $\Omega$ is {\it stationary} in the following sense:
\begin{definition*}[\sc Stationary ensemble] We say that an ensemble on $\Omega$ is stationary if for any shift vector $z\in\R^d$ the random field $A$ and its shifted version $A(\cdot+z)\colon x\mapsto A(x+z)$ have the same distribution.
In other words, for any (integrable) function $\zeta\colon\Omega\rightarrow\mathbb{R}$ 
(which we think of as a random variable) 
we have that $A\mapsto \zeta(A(\cdot+z))$ and $\zeta$ have the same expectation:
\begin{equation}\label{stat}
\en{ \zeta(A(\cdot+z))}=\en{\zeta}.
\end{equation}
\end{definition*}

As before, $\en{ \cdot }$ denotes expectation w.r.t. the ensemble on $\Omega$. Assuming the ensemble is stationary and ergodic in a quantitative way (see below for precise definitions), we prove a Central Limit Theorem-type scaling of the variance of the homogenized coefficients in terms of the system volume $L^d$
\begin{equation}\nonumber
\en{ \left(e_0 \cdot A_{hom} e_1 -\en{ e_0 \cdot A_{hom} e_1}\right)^2}
\le C L^{-d}, 
\end{equation}
for any $\a{e_0},\a{e_1} \le 1$, where $C$ does not depend on $L$ or choice of $e_0,e_1$. 
This estimate is a consequence of the moment bounds on the gradient of the corrector, which we obtain using two different methods. 
Let us now precisely state the results:

\begin{theorem}\label{thm1}
 Let $\alpha \in (0,1]$, $H > 0$, and let $\en{\cdot}$ be stationary and such that $\en{\cdot}$-a.e. coefficient field $A$ is locally $\alpha$-H\"older continuous with constant $H$, meaning that for all $x,y \in B_1$
 \begin{equation}\label{Asmooth}
  \a{A(x) - A(y)} \le H \a{ x-y }^\alpha.
 \end{equation}
Moreover, we assume the ensemble satisfies the following form of the Logarithmic Sobolev inequality: there exists constant $\rho$ such that for any random variable $\zeta \in L^2(\Omega)$, for which the right-hand side in \eqref{1LSI} makes sense, we have
 \begin{equation}\label{1LSI}
  \en{ \zeta^2 \log \zeta^2 } - \en{\zeta^2} \log \en{\zeta^2} \le \V \en{ \int_{\T} \p{ \int_{B_1(z)} \a{\frac{\partial \zeta}{\partial A(y)}} \dy }^2 \dz },
 \end{equation}
 where $\frac{\partial \zeta}{\partial A(\cdot)} \in L^1(\T;\mathcal{L}(Y^d,Y^d))$ denotes the G\^ateaux derivative 
in the sense that 
\begin{equation}\label{frechet}
 \int_{\T} \frac{\partial \zeta}{\partial A(y)} B(y) \dy = \lim_{\delta \to 0+} \frac{\zeta\p{A+\delta B} - \zeta(A)}{\delta}
\end{equation}
for any sufficiently {\it smooth} $B \in L^\infty\p{\T;\mathcal{L}\p{Y^d,Y^d}}$ such that $A + \delta B \in \Omega$ for sufficiently small $\delta > 0$. 
Then for any $e \in Y^d, \a{e} \le 1$ and any $2 \le p < \infty$ we have
\begin{equation}\label{1moment}
 \en{ \a{ \nabla \phi(0) + e}^p }^\frac{1}{p} \le C\exp(C\frac{p-2}{\rho}),
\end{equation}
where the corrector $\phi$ is the unique meanfree solution to $-\nabla \cdot A (\nabla \phi + e ) = 0$. For any $e_i \in Y^d$, $\a{e_i} \le 1$, $i=0,1$ we estimate variance of the homogenized coefficient
\begin{equation}\label{1error}
\en{ \left(e_0 \cdot A_{hom}e_1-\en{ e_0 \cdot A_{hom}e_1}\right)^2}
\le C(d,n,\lambda,\alpha,H,\rho) L^{-d}.
\end{equation} 
\end{theorem}

Unless stated otherwise, here and in the sequel $C$ stands for a generic constant that only depends on the ellipticity ratio $\lambda$, dimension $d$, number of equations $n$, 
$\alpha,H$ related to the smoothness of coefficients, and possibly constant $\rho$ in the (LSI) or (SG).

The proof of Theorem \ref{thm1} is based on Green function representation. To show \eqref{1moment}, by Schauder theory it is enough to estimate $\en{ \zeta^p }$ for $\zeta = \p{ \int_{B_1} |\nabla \phi(x) + e|^2 \dx }^\oh$. We observe that $p \frac{\ud}{\ud p} \en{ \zeta^p } = \en{ \zeta^p \ln \zeta^p }$ and estimate the right-hand side using \eqref{1LSI} (so called Herbst argument). To estimate the right-hand side of \eqref{1LSI} for $\zeta^p$, we express $\frac{\partial \zeta}{\partial A(y)}$ using Green function. To deal with this we estimate the $L^2$-norm of the second mixed derivative of the Green function away from its singularity. In the end we obtain a differential inequality for $F(p) := \en{ \zeta^p }$ in the form $p \frac{\ud}{\ud p} F(p) \le F(p) \ln F(p) + CF(p)$, which together with the fact that by simple energy estimates $F(2) \le C$ gives (exponential) control on $F(p)$. To show \eqref{1error}, we appeal to the Spectral Gap estimate for $\zeta := e_0 \cdot A_{hom} e_1$ and use bound on $\en{ \a{ \nabla \phi(0) + e }^4 }$. 

Let us now provide one example of a stationary ensemble on $\Omega$, which satisfies assumptions of Theorem \ref{thm1} in the case $d=n$, and is of particular practical interest:

\begin{example}[Random linear elasticity]
 We consider an elastic material with stiffer inclusions positioned randomly in the sample. 
%
 More precisely, let the torus size $L$ be an integer, $0 < \lambda < \frac{1}{2}$, and let $\eta : \T \to [0,1]$ be a Lipschitz mask with support in $B_1$ ($\eta$ describes the stiffness of the inclusion). For each grid point $z \in \ZL$, let $s(z) \in [0,1)^d$ be a random shift vector, chosen independently and uniformly. In other words, consider $\bar\Omega := \p{[0,1)^d}^{\ZL}$ with a probability measure $m$ obtained as a product of uniform measures on $[0,1)^d$. It is known that the uniform (Lebesgue) measure on the unit cube satisfies (LSI), i.e., for any $\bar\zeta \in C^1([0,1)^d)$ with $\int_{[0,1)^d} \bar\zeta^2 \dx = 1$ we have 
\begin{equation}
 \int_{[0,1)^d} \bar\zeta^2 \log \bar\zeta^2 \dx \le \frac{1}{2\rho} \int_{[0,1)^d} \a{ \nabla \bar\zeta }^2 \dx
\end{equation}
(it is also known that $\pi \le \rho \le \pi^2$, see, e.g. \cite{compris}). By the tensorization property (see, e.g., book by Ledoux \cite[Proposition 4.4]{ledouxbook}) the measure $m$ also satisfies (LSI) with the same constant $\rho$. 

 
 We use $m$ to define a probability measure on $\Omega$. Since in this example $d=n$, we can identify $Y^d$ with square matrices $\R^{d \times d}$. 
 For an element of $\bar \Omega$ we define $A$ using 
 \begin{equation}\label{defA}
  A(x)y := \max\p{2\lambda,\max_{z \in \ZL} \eta(x-(s(z)+z))} \sym y,
 \end{equation}
 where $\sym y$ denotes the symmetric part of the matrix $y$. We see from \eqref{defA} that $y  \cdot A y \ge 2\lambda \a{ \sym y }^2$, and using Korn's inequality we get that $A$ satisfies the first condition in \eqref{ellipticity}. The boundedness of $A$ (second condition in \eqref{ellipticity}) trivially follows from \eqref{defA}. 
  
 We observe that the push-forward of $m$ using the above-defined map defines a stationary probability measure on $\Omega$. Since $m$ satisfies (LSI), we get for any (smooth) $\zeta \in L^2(\Omega)$, $\en{ \zeta^2 }=1$:
 \begin{equation}\nonumber
  \begin{aligned}
  \en{ \zeta^2 \log \zeta^2 } &= \int_{\bar \Omega} \bar\zeta^2 \log \bar\zeta^2 \ud m \overset{\textrm{LSI for } m}\le \frac{1}{2\rho} \int_{\bar \Omega} \sum_{z\in\ZL} \a{ \frac{ \partial \bar\zeta}{\partial s(z)} }^2 \ud m  
  \\ &\lesssim \en{ \sum_{z \in \ZL} \p{ \int_{B_{1+\sqrt{d}}(z)} \a{ \frac{\partial \zeta}{\partial A(x)} } \dx }^2 } 
  \\ &\lesssim \en{ \int_{\T} \p{ \int_{B_{2+\sqrt{d}}(z)} \a{ \frac{\partial \zeta}{\partial A(x)} } \dx }^2 \dz }
  \lesssim \en{ \int_{\T} \p{ \int_{B_1(z)} \a{ \frac{\partial \zeta}{\partial A(x)} } \dx }^2 \dz },
  \end{aligned}
 \end{equation}
 where $\bar \zeta$ is a composition of $\zeta$ with the map defined in \eqref{defA}. We showed that this model satisfies all assumptions of Theorem \ref{thm1}.  
\end{example}

In Theorem \ref{thm2} we obtain a similar conclusion as in Theorem \ref{thm1}, but under weaker assumptions and using Green function-free approach. In fact, we do not assume any smoothness assumption on coefficient fields $A \in \Omega$, and instead of Logarithmic Sobolev inequality \eqref{1LSI} we assume that the ensemble satisfies  Spectral Gap inequality \eqref{SG}, which is weaker than \eqref{1LSI}. To simplify the exposition we assume the torus size $L$ is an integer:

\begin{theorem}\label{thm2}
There exists $q=q(\lambda,d)$, $1 < q < 2$ with the following property. Let $\en{\cdot}$ be  stationary and satisfies the following form of the Spectral Gap inequality: there exists $\rho > 0$ such that for any $\zeta \in C_b(\Omega)$ (bounded and continuous w.r.t. the $L^\infty$-topology on $\Omega$)
 \begin{equation}\label{SG} 
  \en{ \p{ \zeta - \en{\zeta} }^2 } \le \frac{1}{\rho} \en{ \p{ \sum_{z \in \ZL} \p{ \osc_{B_\d(z)} \zeta}^q }^{\frac{2}{q}} }.
 \end{equation}
 Then
\begin{equation}\label{2moment}
 \en{ \p{ \int_{B_{\sqrt{d}}} |\nabla \phi + e|^2}^p } \le C(d,n,\lambda,p,\rho).
\end{equation}
If the ensemble satisfies \eqref{SG} with $q=2$, then also
\begin{equation}\label{2error}
\en{ \left(e_0\cdot A_{hom}e_1-\en{ e_0\cdot A_{hom}e_1}\right)^2}
\le C(d,n,\lambda,\rho) L^{-d}
\end{equation} 
for any $e_i \in Y^d$, $\a{e_i} \le 1$, $i=0,1$. 
\end{theorem}

Similarly to \eqref{osc}, the oscillation in \eqref{SG} is defined through
\begin{multline}\nonumber
 \osc_{B_{\sqrt{d}}(z)} \zeta(A) := \sup\left\{ \zeta(\tilde A)\ |\ \tilde A \in \Omega \textrm{ s.t. } \tilde A(x) = A(x)\ \forall x \not\in B_{\sqrt{d}}(z) \right\}
\\ - \inf\left\{ \zeta(\tilde A)\ |\ \tilde A \in \Omega \textrm{ s.t. } \tilde A(x) = A(x)\ \forall x \not\in B_{\sqrt{d}}(z) \right\}.
\end{multline}


\begin{remark}
Since for any non-negative sequence $\{ a_k \}$ and $p \ge q \ge 1$ we have $\p{ \sum_k a_k^p }^{\frac{1}{p}} \le \p{ \sum_k a_k^q }^{\frac{1}{q}}$, the right-hand side in \eqref{SG} decreases as $q$ increases. From that we see that the assumption \eqref{SG} gets stronger as $q$ becomes larger, in particular Theorem \ref{thm2} holds for any stationary ensemble that satisfies \eqref{SG} with $q=2$, which would be a more common form of the Spectral Gap inequality. 
\end{remark}

\begin{remark}\label{remarkq}
 If we assume that all coefficient fields $A \in \Omega$ satisfy the following stronger (local) version of \eqref{ellipticity}:
 \begin{equation}\label{strong}
\forall r > 0, \varphi \in W^{1,2}(B_r,Y): \quad \int_{B_r} \nabla \varphi \cdot A \nabla \varphi \ge \lambda \int_{B_r} \a{ \nabla \varphi }^2, 
 \end{equation}
 then the best (smallest) $q$ in \eqref{SG} for which Theorem \ref{thm2} holds, satisfies $q(\lambda,d) \to 1$ as $\lambda \to 1$.
\end{remark}

Let us now sketch the main steps in the proof of \eqref{2moment} (for simplicity assuming \eqref{SG} holds for $q=2$). Let $R \ge 1$ be fixed. Then:
\begin{itemize}
 \item there exists $\bar \alpha(\lambda,d)>0$ such that $u$, the unique meanfree periodic solution of $- \nabla A \nabla u = \nabla \cdot g$ in $\T$, satisfies $\int_{B_R} \a{\nabla u}^2 \lesssim \int_{\T} \p{ \frac{|x|}{R} + 1 }^{\bar \alpha} |g|^2$;
 \item for any $F$, deterministic linear functional on $L^2(B_{2R})$, that for any $f \in W^{1,2}(B_{2R})$ satisfies $\a{F(\nabla f)}^2 \lesssim \int_{B_{2R}} \a{ \nabla f }^2$, we have 
 \begin{equation}\nonumber
 \sum_{z \in \ZL} \p{ \osc_{B_{\sqrt{d}}(z)} F(\nabla \phi + e)} ^{2} \lesssim \sup_{z \in \ZL} \p{ \frac{|z|}{R} + 1}^{-\alpha} \int_{B_{\sqrt{d}}(z)} \a{\nabla \phi + e}^2
\end{equation} 
 (exponent $\alpha$ is slightly smaller than $\bar \alpha$ from the previous step);
\item we find a finite deterministic collection of $F_k$ for which the previous step applies, and such that 
 \begin{equation}\nonumber
  \en{ \p{ \int_{B_R} \a{ \nabla \phi + e}^2 }^p } \le C^p \max_k \en{ \a{ F_k(\nabla \phi + e ) }^{2p} };
 \end{equation}
 \item we show $\int_{B_{\sqrt{d}}} |\nabla \phi + e|^2 \lesssim R^{-\alpha} \int_{B_R} |\nabla \phi + e|^2$, and use two previous steps in the $p$-version of the Spectral Gap inequality to show
\begin{equation}\nonumber
 \en{ \p{ \int_{B_R} \a{ \nabla \phi + e}^2 }^p } \lesssim C^{p} \p{ R^{dp} + R^{d-p\alpha}\en{ \p{ \int_{B_R} \a{\nabla \phi + e}^2}^p } }.
\end{equation}
 Hence for $p\alpha > d$ and $R$ large enough we get the desired bound $\en{ \p{ \int_{B_R} \a{ \nabla \phi + e}^2 }^p } \le C$. 
\end{itemize}

We now provide one example of an ensemble, 
the {\it Poisson ensemble}, which satisfies assumptions of Theorem \ref{thm2}:
\begin{example}
Let the configuration of points $X:=\{X_i\}_{i=1,\cdots,N}$ on the torus
be distributed according to the Poisson
point process with density one. This means the following:
\begin{itemize}
\item 
For any two disjoint (Lebesgue measurable) subsets $D$ and $D'$ of the torus we have that
the configuration of points in $D$ and the configuration of points in $D'$ are independent.
%
\item For any (Lebesgue measurable) subset $D$ of the torus, 
the number of points in $D$ is Poisson distributed;
the expected number is given by the Lebesgue measure of $D$.
\end{itemize}
Note that $N$ is random, too.

\smallskip

With any realization $X=\{X_i\}_{i=1,\cdots,N}$ of the Poisson point process, 
we associate the coefficient field $A\in\Omega$ via
\begin{equation}\nonumber
A(x)=\left\{\begin{array}{ccc}
\lambda&\mbox{if}&x\in\bigcup_{i=1}^{N}B_1(X_i)\\
1&\mbox{else}\end{array}\right\}\mathrm{Id}.
\end{equation}
This defines an ensemble on $\Omega$ by ``push-forward'' of the Poisson Point Process.  It is easy to see that the ensemble is stationary and $A \in \Omega$ satisfy \eqref{ellipticity}. 
Moreover, since the Poisson point process satisfies the Spectral Gap inequality (see (1.8) in \cite{SG4Poisson} or \cite{SG4Poisson2}), \eqref{SG} with $q=2$ holds as well. 
\end{example}

\section{Simple argument for moment bounds in the discrete setting}\label{sec:disc}

The proof of Theorem~\ref{thm1} is based on Green function representation, and a similar idea can be used to obtain moment bound on the gradient of the corrector also in the discrete setting. In the case of a whole space $\Z^d$ and for an equation with massive term, Ben-Artzi, Marahrens, and Neukamm \cite{AMN} used this idea to obtain moment bounds on $\nabla \phi$. In fact, their main achievement are moment bounds on the corrector itself in the case of non-diagonal coefficient; moment bounds on $\nabla \phi$ are only small part of their work.

In \cite{AMN} they use the following form of (LSI): for all random variables $\zeta \in C_b(\Omega)$
\begin{equation}\label{3.1}
 \en{ \zeta^2 \ln \frac{\zeta^2}{\en{\zeta^2}} } \le \V \en{ \sum_{y \in \Z^d} \p{ \osc_{a(y)} \zeta}^2 },
\end{equation}
where $\osc_{a(y)} \zeta$ is defined in \eqref{osc}. In order to obtain better control of the constants in the estimate via Herbst argument, we consider \eqref{3.1} with $\osc_{a(e)} \zeta$ on the right-hand side replaced by $\frac{\partial \zeta}{\partial a(e)}$. In the discrete setting, the coefficient field $a$ is a function on $\Z^d$ with values in $\R^{d \times d}$. Assuming only the {\it lower bound} on the coefficients in the form of $y \in \ZL, \forall v \in \R^d : v \cdot a(y) v \ge \lambda \a{v}^2$, we prove that any moment of $\nabla \phi(0)$ is controlled by $\en{ \a{ \nabla \phi(0) + \xi }^2 }^{\frac{1}{2}}$. We point out that to estimate $\en{ \a{ \nabla \phi(0) + \xi }^2 }^{\frac{1}{2}}$ one has to assume also a (not necessarily pointwise) upper bound on $a$. 



Before we state the precise statement we briefly introduce the discrete setting. To simplify the exposition, here we consider only {\it scalar} equations on $\ZL$. For a function $u : \ZL \to \R$, the discrete gradient is a function defined on the set of edges $\E := \bigcup_{i=1}^d (\ZL) + e_i$, where $e_1,\ldots,e_d$ is the canonical basis in $\R^d$, via
\begin{equation}\nonumber
 \nabla u(b) = u(x+e_i) - u(x), \quad \textrm{for } b = [x,x+e_i] \in \E.
\end{equation}
For a function $g : \E \to \R$, the negative divergence $\nabla^* g$  is a function on $\ZL$ defined via
\begin{equation}\nonumber
 (\nabla^* g)(x) = \sum_{i=1}^d g([x-e_i,x])-g([x,x+e_i]).
\end{equation}
In the discrete setting, the coefficient field $a$ associates a uniformly elliptic matrix to each point in $\ZL$. 
For simplicity we only consider the case of diagonal matrices. In that case a coefficient field $a$ can be thought of as a scalar function on the edges $\E$ which satisfies $a(e) \ge \lambda$ for each $e \in \E$. 
As mentioned before we do not assume any upper bound on $a$. 
For differentiable $\zeta : \Omega \to \R$ and $e \in \E$ we write
\begin{equation}\nonumber
 \partial_e \zeta = \frac{\partial \zeta}{\partial a(e)},
\end{equation}
and as in the continuum setting $\en{ \cdot }$ will denote the ensemble average on $\Omega := [\lambda,\infty)^{\E}$. 

\begin{theorem}[Moment bounds on $\nabla \phi$ in the discrete setting]\label{thm3}
 Let $\en{ \cdot }$ be stationary and satisfies the following Logarithmic Sobolev inequality:
 \begin{equation}\label{LSIdisc}
   \en { \zeta^2 \ln \zeta^2 } - \en {\zeta^2} \ln \en{\zeta^2} \le \frac{1}{2\rho} \en{  \sum_{e \in \E}  \a{ \partial_e\zeta}^2 }
 \end{equation}
 for all random variable $\zeta \in L^2(\Omega)$ for which the right-hand side makes sense. Let $\xi \in \R^d, \a{\xi} \le 1$ be fixed and for $a \in [\lambda,\infty)^{\E}$ let  $\phi(a; \cdot)$ be the unique solution to 
 \begin{equation}\label{3-1}
  \nabla^* a(\nabla \phi(a;\cdot) + \xi) = 0
 \end{equation}
 with $\phi(0)=0$. 
 Then for all $1 \le p < \infty$ 
\begin{equation}\label{3s}
 \en{ \sum_{i=1}^d |D_i \phi + \xi_i|^{2p} }^\frac{1}{p} \le \exp\left( \frac{p-1}{2\rho\lambda^2} \right) \en{ \sum_{i=1}^d \left|D_i \phi + \xi_i\right|^2 },
\end{equation}
where $D_i \phi := \nabla \phi(e_i)$. 
\end{theorem}

Compared to our setting, in the case of the whole space $\Z^d$ the existence and uniqueness of the corrector $\phi$ is much more subtle. For example, it can be shown that in any dimension $d \ge 2$ there exists a unique corrector $\phi$ with $\phi(0)=0$ such that $\nabla \phi$ is stationary. In contrast, the corrector itself is stationary only if $d > 2$ (see, e.g., \cite{GM}). 

\begin{proof}[Sketch of the proof]
 \mbox{}\\{\bf Step 1:} We claim 
 \begin{equation}\label{3-2}
  \partial_e(\nabla \phi + \xi)(b) = -\nabla\nabla G(a;b,e)(\nabla \phi+\xi)(e),
 \end{equation}
 where $G(a;x,y)$ is the Green function defined by $\nabla^* a(\cdot) \nabla G(a;\cdot,y) = \delta_y(\cdot)$, and $\nabla \nabla G$ denotes the second mixed derivative. Here $\xi(e) = \nabla\bar\xi(e)$ with $\bar\xi(x) := \xi \cdot x$. 

To prove \eqref{3-2}, we first differentiate \eqref{3-1} w.r.t. $a(e)$, $e \in \E$ to get $\nabla^* a \nabla \partial_e \phi = \nabla^* g_e$, where $g_e(b) = -\delta_{e}(b) \p{\nabla \phi + \xi}(b)$. Using Green function representation we see
 \begin{equation}\nonumber
\partial_e \phi(x) = \sum_{y \in \Z^d} G(a;x,y)(\nabla^* g_e)(y) = \sum_{b \in \E} \nabla G(a;x,b) g_e(b) = -\nabla G(a;x,e) (\nabla \phi + \xi)(e).
 \end{equation}
Hence $\partial_e (\nabla \phi+\xi)(b) = \nabla \partial_e \phi(b) = - \nabla \nabla G(a;b,e)(\nabla \phi +\xi)(e)$, and \eqref{3-2} is proved.

\noindent\smallskip{\bf Step 2} (see also (37) in \cite{MO}): We claim that for any $e \in \E$ we have
 \begin{equation}\label{3-3}
  \sum_{b \in \E} \a{ \nabla \nabla G(a;b,e) }^2 \le \lambda^{-2},
 \end{equation}
 and 
 \begin{equation}\label{3-3sym}
  \sum_{e \in \E} \a{ \nabla \nabla G(a;b,e) }^2 \le \lambda^{-2}.
 \end{equation} 
 For any function $\varphi \in l^2(\Z^d)$ and $y \in \ZL$ we have by definition of Green function $\sum_{b\in\E} \nabla \varphi(b) a(b) \nabla G(a;b,y) = \varphi(y)$. We differentiate the weak formulation in $y$ and set $\varphi(x) = \nabla G(a;x,e)$ 
 to obtain
 \begin{equation}\nonumber
  \sum_{b\in \E} \nabla \nabla G(a;b,e) a(b) \nabla \nabla G(a;b,e) = \nabla \nabla G(a;e,e).
 \end{equation}
 By assumption $a \ge \lambda$, so the left-hand side is bounded from below by $\lambda\sum_{b \in \E} \a{\nabla \nabla G(a;b,e)}^2$ while the right-hand side is trivially bounded from above by $\p{ \sum_{b \in \E} \a{\nabla \nabla G(a;b,e)}^2}^\oh$, and \eqref{3-3} follows.  
 Finally, symmetry of $G$ in the form $\nabla \nabla G(a;e,b) = \nabla \nabla G(a; b,e)$
 (see discussion below (16) in \cite{AMN}) implies \eqref{3-3sym}. 

\noindent\smallskip{\bf Step 3:} We claim
 \begin{equation}\label{3-4}
  \sum_{i=1}^d \en{ \p{ \sum_{e \in \E} \p{ \partial_e \p{ D_i \phi + \xi_i } }^2 }^p } 
  \le \lambda^{-2p} \sum_{i=1}^d \en{ \a{ D_i \phi + \xi_i }^{2p} },
 \end{equation}
 where $D_i \phi = \nabla \phi(e_i)$. 
 We know from Step 1 that $\partial_e (D_i \phi + \xi_i) = -\nabla\nabla G(a;e_i,e)(\nabla \phi + \xi)(e)$. Hence 
 \begin{align*}
  \p{ \sum_{e \in \E} \p{ \partial_e (D_i \phi + \xi_i) }^2 }^{p} =& \p{ \sum_{e \in \E} \p{ \nabla\nabla G(a;e_i,e)}^2 \p{\nabla \phi + \xi}^2(e) }^{p} 
\\ 
\le& \p{ \sum_{e \in \E} \p{ \nabla\nabla G(a;e_i,e)}^2 }^{p-1} \!\!\!\!\sum_{e \in \E} \p{ \nabla\nabla G(a;e_i,e)}^2 \p{\nabla \phi + \xi}^{2p}(e) 
\\ 
\overset{\mathclap{\textrm{\eqref{3-3sym}}}}{\le} &\ \lambda^{-2(p-1)} \sum_{e \in \E} \p{ \nabla\nabla G(a;e_i,e)}^2 \p{\nabla \phi + \xi}^{2p}(e).
 \end{align*}
 We take the ensemble average of the above and use $\E = \bigcup_{j=1}^d {\p{\ZL} + e_j}$ to obtain
 \begin{align*}
  \en{ \p{ \sum_{e \in \E} \p{ \partial_e (D_i \phi + \xi_i) }^2 }^{p} } 
  &\le \lambda^{-2(p-1)} \en{ \sum_{j=1}^d \sum_{x \in \ZL} \p{ \nabla\nabla G(a;e_i,x+e_j)}^2 \p{\nabla \phi + \xi}^{2p}(x+e_j) }
\\ &\overset{\mathclap{\eqref{stat}}}{=} \lambda^{-2(p-1)} \sum_{j=1}^d \sum_{x \in \ZL} \en{ \p{ \nabla\nabla G(a;-x+e_i,e_j)}^2 \p{\nabla \phi + \xi}^{2p}(e_j) }
\\ &= \lambda^{-2(p-1)} \sum_{j=1}^d \en{ \p{ \sum_{x \in \ZL} \p{ \nabla\nabla G(a;-x+e_i,e_j)}^2 } \p{D_j \phi + \xi_j}^{2p} }.
 \end{align*}
 Finally, we sum over $i=1,\ldots,d$ to get
 \begin{multline*}
  \sum_{i=1}^d \en{ \p{ \sum_{e \in \E} \p{ \partial_e (D_i \phi + \xi_i) }^2 }^{p} } 
\\ \le \lambda^{-2(p-1)}\sum_{j=1}^d \en{ \underbrace{\p{ \sum_{i=1}^d \sum_{x \in \ZL} \p{ \nabla\nabla G(a;-x+e_i,e_j)}^2 }}_{\le \lambda^{-2} \textrm{ by \eqref{3-3} in Step 2}} \p{D_j \phi + \xi_j}^{2p} },
 \end{multline*}
 and \eqref{3-4} follows. 

\noindent\smallskip{\bf Step 4:} To prove \eqref{3s} we use variation of Herbst argument (see, e.g, \cite{ledouxbook}), which is based on the identity $q \frac{\ud}{\ud{q}} f^q = f^q \ln f^q$ and the control of $f^q \ln f^q$ using Logarithmic Sobolev inequality. We denote $f_i := D_i \phi + \xi_i$, and observe that for $q \ge 2$
\begin{align*}
 q \frac{\ud}{\ud{q}} \sum_{i=1}^d \en{ \a{f_i}^q } &= \sum_{i=1}^d \en{ \a{f_i}^q \ln \a{f_i}^q } \overset{\eqref{LSIdisc}}{\le} \sum_{i=1}^d \p{ \en{ \a{f_i}^q} \ln \en{ \a{f_i}^q} + \frac{1}{2\rho} \en{ \sum_{e\in\E} \a{ \partial_e \p{ \a{f_i}^{\frac{q}{2}}} }^2 } } 
\\ &\le \sum_{i=1}^d \p{ \en{ \a{f_i}^q } \ln \en{ \a{f_i}^q } + \frac{q^2}{8\rho} \en{ \a{f_i}^{q-2}  \sum_{e\in\E} \a{ \partial_e f_i }^2 } } 
\\ &\le \sum_{i=1}^d \en{ \a{f_i}^q } \ln \en{ \a{f_i}^q } + \frac{q^2}{8\rho} \sum_{i=1}^d \en{ \a{f_i}^q }^{\frac{q-2}{q}} \en{ \p{ \sum_{e\in\E} \a{ \partial_e f_i }^2}^{\frac{q}{2}} }^\frac{2}{q}
\\ &\le \sum_{i=1}^d \en{ \a{f_i}^q } \ln \en{ \a{f_i}^q } + \frac{q^2}{8\rho} \p{ \sum_{i=1}^d \en{ \a{f_i}^q }}^{\frac{q-2}{q}} \p{ \sum_{i=1}^d \en{ \p{ \sum_{e\in\E} \a{ \partial_e f_i }^2}^{\frac{q}{2}} }}^\frac{2}{q} 
\\ &\!\! \overset{\eqref{3-4}}{\le} \p{ \sum_{i=1}^d \en{ \a{f_i}^q } } \ln \p{ \sum_{i=1}^d \en{ \a{f_i}^q } }+ \frac{q^2}{8\rho\lambda^2} \p{ \sum_{i=1}^d \en{ \a{f_i}^q }}.
\end{align*}
Writing $F(q) := \sum_{i=1}^d \en{ \a{f_i}^q }$, the last inequality yields $q \frac{\ud}{\ud q} F(q) \le F(q) \ln F(q) + \frac{q^2}{8\rho\lambda^2}F(q)$. This is equivalent to $\frac{\ud}{\ud q}\p{\frac{1}{q} \ln F(q)} \le \frac{1}{8\rho\lambda^2}$, and so after integration we obtain \eqref{3s} in the form $F(2p)^{\frac{1}{p}} \le \exp(\frac{p-1}{2\rho\lambda^2})F(2)$.

\end{proof}

\section{The main ingredients in the proof of Theorem~\ref{thm1}}\label{sec:pf1}

The proof of Theorem \ref{thm1} follows the idea of the proof of Theorem \ref{thm3}. First, we obtain the $L^2$-estimate for the second mixed derivative of the Green matrix. In the continuum setting $\gg G(A;x,y)$ (at least for smooth $A$) behaves like $|x-y|^{-d}$. Hence it is not $L^2$-integrable near the singularity $x=y$, and one should not expect a complete analogue of \eqref{3-3}. Instead, we obtain:

\begin{lemma}\label{lm1}
 Let the coefficient field $A$ satisfies~\eqref{ellipticity} and \eqref{Asmooth}. Then for every $R \in (0,1)$ and every $y \in \T$ we have
\begin{equation}\label{eq1.1}
 R^d \int_{
|x-y| > R} \left| \gg G(x,y)\right|^2 \dx \le C(d,n,\lambda,\alpha,H).
\end{equation}
\end{lemma}

 If $A$ satisfies \eqref{ellipticity} and \eqref{Asmooth}, then so does its transpose $A^t$ (in coordinates $(A^t)_{ij}^{\alpha\beta} = A_{ji}^{\beta\alpha}$), and estimate \eqref{eq1.1} holds also for Green matrix $G^t$ associated with $A^t$. By \cite[Theorem 1]{DMGreen} $G^t(y,x) = G(x,y)$, and so \eqref{eq1.1} for $G^t$ implies 
 \begin{equation}\label{eq1.1+}
  R^d \int_{
|x-y| > R} \left| \gg G(x,y)\right|^2 \dy \le C(d,n,\lambda,\alpha,H).
 \end{equation}
Here, given a coefficient field $A$, the Green matrix $G : \T \times \T \rightarrow Y^n$ is a mean-free $L$-periodic function which satisfies
\begin{equation}\nonumber
 -\nabla \cdot \left( A(\cdot) \nabla G^k(\cdot,y) \right) = \p{ \delta_y(\cdot) - L^{-d}} e_k
\end{equation}
for $k=1,\ldots,n$, where $e_1,\ldots,e_n$ denotes the canonical basis in $Y$. Existence of the Green matrix follows from works of Fuchs \cite{fuchs84,fuchs86} and also  Dolzmann and M\"uller \cite{DMGreen}). In these papers existence and properties of Green matrix for a bounded domain with zero boundary data are proved, but their methods apply also in the periodic setting. They need the coefficient field $A$ to be continuous (or at least to belong to $L^\infty$, and in addition to the space of functions with vanishing mean oscillations $VMO$ in the case $d\ge3$), $A$ has to satisfy Legendre-Hadamard condition and be coercive in the sense of \eqref{ellipticity}. Recently, Conlon, Giunti and Otto \cite{CGO} proved existence of the Green matrix (almost surely) for random coefficient field $A$ assuming only  \eqref{ellipticity} and stationarity of the ensemble. 

As an immediate consequence of Lemma~\ref{lm1} we get
\begin{cor}\label{cor1}
 Let $A$ satisfies~\eqref{ellipticity} and \eqref{Asmooth}. Then for every $\ep > 0$ 
we have
\begin{equation}\label{eq1.2}
 \begin{aligned}
 \forall y \in \T: \quad \int_{
|x-y| < 4} |x-y|^{d+\ep} \left| \gg G(x,y)\right|^2 \dx &\le C(d,n,\lambda,\alpha,H,\ep),\\
 \forall x \in \T: \quad \int_{
|x-y| < 4} |x-y|^{d+\ep} \left| \gg G(x,y)\right|^2 \dy &\le C(d,n,\lambda,\alpha,H,\ep).
 \end{aligned}
\end{equation}
 
\end{cor}

Using the corollary we show an analogue of \eqref{3-4}:
\begin{lemma}\label{lm2}
 Let $\zeta := \left( \int_{B_1} (u(x) - u(0))^2 \dx \right)^{1/2}$, where $u(x) := e \cdot x + \phi(x)$. Then under the assumption of Theorem \ref{thm1} we have for any $1 \le p < \infty$
 \begin{equation}\label{LSIest}
  \en{ \left( \int_{\T} \left( \int_{B_1(z)} \left| \frac{\partial \zeta}{\partial A(y)}\right| \dy \right)^2 \dz \right)^p }^{\frac{1}{p}} \lesssim \en{ \zeta^{2p} }^{\frac{1}{p}},
 \end{equation}
where $\frac{\partial \zeta}{\partial A(y)}$ was defined in \eqref{frechet} and $\lesssim$ denotes $\le$ up to a multiplicative constant depending on $d,n,\lambda,\alpha,H$.
\end{lemma}

In order to prove Lemma \ref{lm2} we will need to show 
\begin{equation}\label{dudA}
 \frac{\partial u(x) }{\partial A(y)} = -\nabla_y G(x,y) \nabla u(y).
\end{equation}
Having Lemma \ref{lm2}, it is straightforward to use almost the same argument as in Step 4 of the proof of Theorem \ref{thm3} to estimate any moment of $\zeta$. Since $A$ is H\"older continuous, using Schauder estimates we can control $\a{\nabla u(0)} = |\nabla \phi(0) + e|$ with $\zeta$ and \eqref{1moment} follows. Finally, since the Spectral Gap inequality follows from the Logarithmic Sobolev inequality, \eqref{1moment} with $p=4$ implies \eqref{1error}. 



%

\section{Proofs for Theorem 1}\label{sec:pf1+}

\begin{proof}[Proof of Lemma \ref{lm1}]
We split the proof into two step. The first one resembles the first step in the proof of Theorem~\ref{thm3}. The idea in the discrete setting was to differentiate the equation for the Green function (in its weak form) in the $y$ variable, then test the equation with the gradient of the Green function itself, and use ellipticity to obtain the estimate for the $L^2$ norm of the second mixed derivative of $G$. In the continuum setting, because of the singularity of $G$, this can not be repeated verbatim. Instead, we apply the argument for a smoothed-out version of $G$ (Step 1), and then use Schauder regularity theory to estimate the difference between $\gg G$ and its mollification (Step 2). 

\mbox{}\\
\noindent\medskip 
{\bf Step 1:}
We smear out $\gg G$ in y over lengthscale $r \in (0,1)$ and estimate its $L^2$-norm by a multiple of $r^{-\frac{d}{2}}$. More precisely, for any $r \in (0,1)$ and $y \in \T$ we claim
\begin{equation}\label{est1}
 \int_{\T} \left| \frac{1}{\a{B_r}} \int_{|y' - y| \le r} \gg G(x,y') \ud{y'}\right|^2 \dx \le \lambda^{-2}\frac{nd}{\a{B_r}}.
\end{equation}
Indeed, fix $k \in \{ 1,\ldots,n \}$. By definition of Green matrix $G$ we have for any mean-free periodic test function  $\varphi \in W^{1.2}(\T;Y)$ 
\begin{equation}\nonumber
  \int_{\T} \nabla_x \varphi(x) \cdot A(x) \nabla_x G^k(x,y) \dx = \varphi(y).
\end{equation}
We apply convolution in $y$ with a kernel $\a{B_r}^{-1} \chi_{B_r(0)}$ to both sides of the 	equation and differentiate in $y_i$ to get 
\begin{equation}\nonumber
 \int_{\T} \nabla_x \varphi(x) \cdot A(x) \nabla_x \nabla_{y_i} G_r^k(x,y) \dx = \nabla_{y_i} \varphi_r(y),
\end{equation}
where subscript $f_r$ denotes average value of $f$ in $y$ over a ball of radius $r$. By approximation we can set $\varphi(x) := \nabla_{y_i} G^{k}_{r}(x,y)$ and use \eqref{ellipticity} to get
\begin{equation}\nonumber
 \lambda \int_{\T} \a{ \nabla_x \nabla_{y_i} G^k_r(x,y)}^2 \dx \le  \int_{\T} \nabla_x \nabla_{y_i} G_r^k(x,y) \cdot A(x) \nabla_x \nabla_{y_i} G_r^k(x,y) \dx = \nabla_{x_i} \nabla_{y_i} G^{k}_{rr} (y,y),
\end{equation}
where subscript $_{rr}$ denotes averaging in both variables. 
Using $\| f * g \|_{L^\infty} \le \| f \|_{L^2} \| g \|_{L^2}$ we obtain
\begin{equation}\nonumber
 \left| \nabla_{x_i} \nabla_{y_i} G_{rr}^{k} (y,y) \right| 
\le \a{B_r}^{-1/2} \left( \int_{\T} \left| \nabla_{x_i} \nabla_{y_i} G_{r}^{k}(x,y) \right|^2 \dx \right)^{1/2},
\end{equation}
which, combined with the previous relation and after summing over $k$ and $i$, implies~\eqref{est1}. 

\smallskip\noindent{\bf Step 2:} 
In the next step we improve \eqref{est1} by removing averaging over balls $B_r$ while staying away from the singularity. 

Let $R \in (0,1]$ and $y \in \T$ be fixed. Since we assume $R \le 1$, we can use standard Schauder estimates (see, e.g., \cite[Theorem 5.19]{giaqmart}). Hence, \eqref{Asmooth} together with the fact that away from the singularity $\nabla_y G$ solves $-\nabla_x \cdot A \nabla_x (\nabla_y G) = 0$ implies for all $x \in \T : |x-y| > 2R \ge 2r$:
\begin{multline}\label{est2}
 \left| \frac{1}{\a{B_r}} \int_{|y' - y| < r} \gg G(x,y') \ud{y'} - \gg G(x,y) \right| \le r^{\alpha} \left[ \gg G(x,\cdot) \right]_{\alpha,B_r(y)} 
\\ \le 
C \left( \frac{r}{R} \right)^\alpha \left( \frac{1}{\a{B_{2R}}} \int_{|y' - y| < 2R} |\gg G(x,y')|^2 \ud{y'} \right)^\oh.
\end{multline}
Here, $[\cdot]_{\alpha,B}$ denotes $C^{0,\alpha}$ H\"older norm in a ball $B$. By triangle inequality we get
\begin{multline}\nonumber
 \left( \int_{|x-y| > 3R} |\gg G(x,y)|^2 \dx \right)^\oh 
\\
\begin{aligned}
\le& \left( \int_{|x-y|>3R} \left| \frac{1}{\a{B_r}} \int_{|y' - y| < r} \gg G(x,y') \ud{y'} \right|^2 \dx \right)^\oh
\\ &+ \left( \int_{|x-y|>3R} \left| \frac{1}{\a{B_r}} \int_{|y' - y| < r} \gg G(x,y') \ud{y'} - \gg G(x,y) \right|^2 \dx \right)^\oh 
\\ &\!\!\!\!\!\!\!\!\!\!\!\!\!\overset{\eqref{est1},\eqref{est2}}{\lesssim} \frac{1}{\a{B_r}^\oh} + \left( \frac{r}{R} \right)^\alpha \left( \int_{|x-y|>3R} \frac{1}{\a{B_{2R}}} \int_{|y' - y| < 2R}  |\gg G(x,y')|^2 \ud{y'} \dx \right)^\oh 
\\ \lesssim& \frac{1}{\a{B_r}^\oh} + \left(\frac{r}{R}\right)^{\alpha} \left( \sup_{y' \in \T} \int_{|x-y'|>R} \left| \gg G(x,y') \right|^2 \dx \right)^\oh,
 \end{aligned}
\end{multline}
where the last inequality follows from the inclusion $\{ (x,y') : |x-y| > 3R, |y'-y| < 2R \} \subset \{ (x,y') : |x-y'| > R, |y-y'| < 2R\}$. 
Now consider 
\begin{equation}\nonumber
 \Lambda :=  \sup_{0 < R \le 1} \sup_{y' \in \T} \left( R^d \int_{|x-y'|>R} \left| \gg G(x,y')\right|^2 \dx \right)^\oh.
\end{equation}
We observe that $\Lambda < \infty$, and so the derivation above implies
\begin{equation}\nonumber
 3^{-d} \left( (3R)^d \int_{|x-y| > 3R} |\gg G(x,y)|^2 \dx \right)^\oh \le C \left( \frac{R}{r} \right)^\frac{d}{2} + C \left( \frac{r}{R} \right)^\alpha \Lambda. 
\end{equation}
We choose $r := \ep R$, and take supremum over $0 < R \le 1/3$ and over $y \in \T$ to derive
\begin{equation}\nonumber
 \Lambda \lesssim \left( \frac{1}{\ep} \right)^\frac{d}{2} + \ep^{\alpha} \Lambda,
\end{equation}
which by suitable choice of $\ep$ implies \eqref{eq1.1}.
\end{proof}

\begin{proof}[Proof of Lemma~\ref{lm2}]
 To obtain sensitivity estimate on $\zeta$, by a simple argument we first show that it is enough to understand sensitivity of $u(x) - u(0)$ in the form \eqref{eq100} (Step 1). Then we derive an expression for $\frac{\partial \phi(x)}{\partial A(y)}$ in terms of the Green function (Step 2). In the last step we use this formula (together with conclusions of Corollary \ref{cor1}) to finish the argument. 

 \mbox{}\\
 {\bf Step 1:} To show the sensitivity estimate \eqref{LSIest}, it is enough to prove that for every $x \in B_1(0)$ 
 \begin{equation}\label{eq100}
  \en{ \left( \int_{\T} \left( \int_{B_1(z)} \left| \frac{\partial (u(x) - u(0))}{\partial A(y)}\right| \dy \right)^2 \dz \right)^p } \le C^p \en{ |\nabla u(0)|^{2p} }.
 \end{equation} 
 Recall that $\zeta = \left( \int_{B_1} (u(x) - u(0))^2 \dx \right)^\frac{1}{2}$. We first estimate the inner integral in \eqref{LSIest}: for any $z \in \T$ we have
 \begin{multline}\nonumber
\begin{aligned}
  \int_{B_1(z)} \left| \frac{\partial \zeta}{\partial A(y)}\right| \dy
&\le \int_{B_1(z)} \frac{1}{\zeta} \left( \int_{B_1(0)} |u(x) - u(0)|\cdot\left| \frac{\partial (u(x)-u(0))}{\partial A(y)} \right| \dx \right) \dy 
\\
 &= \frac{1}{\zeta} \int_{B_1(0)} |u(x) - u(0)|\cdot \left( \int_{B_1(z)} \left| \frac{\partial (u(x)-u(0))}{\partial A(y)} \right| \dy \right) \dx 
\\ &\!\!\!\!\overset{\textrm{H\"older}}{\le} \frac{1}{\zeta} \left( \int_{B_1(0)} |u(x) - u(0)|^2 \dx \right)^\oh 
\\
&\qquad \qquad \times \p{ \int_{B_1(0)} \left( \int_{B_1(z)} \left| \frac{\partial (u(x)-u(0))}{\partial A(y)} \right| \dy \right)^2 \dx }^\oh 
\\ &= \p{ \int_{B_1(0)} \left( \int_{B_1(z)} \left| \frac{\partial (u(x)-u(0))}{\partial A(y)} \right| \dy \right)^2 \dx }^{\frac{1}{2}}.
\end{aligned}
 \end{multline}
We use this estimate together with \eqref{eq100} to show \eqref{LSIest}:
 \begin{multline}\nonumber
    \en{ \left( \int_{\T} \left( \int_{B_1(z)} \left| \frac{\partial \zeta}{\partial A(y)}\right| \dy \right)^2 \dz \right)^p } 
\\ \begin{aligned}
\\ &\le \en{ \left( \int_{\T} \int_{B_1(0)} \left( \int_{B_1(z)} \left| \frac{\partial (u(x)-u(0))}{\partial A(y)} \right| \dy \right)^2 \dx \dz \right)^p }
\\ &\le C^p \en{ \int_{B_1(0)} \left( \int_{\T} \left( \int_{B_1(z)} \left| \frac{\partial (u(x)-u(0))}{\partial A(y)} \right| \dy \right)^2 \dz \right)^p \dx }
\\ &=  C^p \int_{B_1(0)} \en{ \left( \int_{\T} \left( \int_{B_1(z)} \left| \frac{\partial (u(x)-u(0))}{\partial A(y)} \right| \dy \right)^2 \dz \right)^p } \dx \overset{\eqref{eq100}}{\le} C^p \en{ |\nabla u(0)|^{2p} }.
   \end{aligned}
 \end{multline}
Using Schauder theory (see, e.g., \cite[Theorem 5.19]{giaqmart}) we get that $|\nabla u(0)| \lesssim \zeta$, and we see that \eqref{eq100} indeed implies \eqref{LSIest}.

\smallskip\noindent{\bf Step 2:} To prove \eqref{eq100} we need to show the following formula for the vertical derivative of the corrector:
\begin{equation}\label{dphiA}
 \frac{\partial \phi(x) }{\partial A(y)} = -\nabla_y G(x,y) \p{ \nabla \phi(y) + e}.
\end{equation}
Here, $\phi$ is the solution of  $-\nabla \cdot A (\nabla \phi + e) = 0$, and $\frac{\partial \phi(x)}{\partial A(y)}$ is understood in the sense of \eqref{frechet}. Since $A$ is $4$-th order tensor and $\phi$ is $1$-st order tensor, $\frac{\partial \phi(x)}{\partial A(y)}$ is $5$-th order tensor (i.e., there is no contraction on indices on the right-hand side). Formally, differentiating equation for $\phi$ with respect to $A(y)$ one gets $-\nabla \p{ A \p{ \nabla \frac{\partial \phi(x)}{\partial A(y)} } } = \delta_y(x) \nabla\p{ \textrm{Id}\p{ \nabla \phi(x) + e} }$, which using Green function representation yields \eqref{dphiA}. 

To prove \eqref{dphiA} rigorously, let $B \in L^\infty(\T;\mathcal{L}(Y^d,Y^d))$ be smooth and $\delta_0 > 0$ be such that $A+\delta B \in \Omega$ for $\delta \in (0,\delta_0)$. 
By $\phi_\delta$ we denote the solution of 
\begin{equation}\label{phibar}
 -\nabla \cdot ( (A+\delta B)(\nabla \phi_\delta + e) ) = 0.
\end{equation}
Subtracting equations for $\phi$ and $\phi_\delta$, we arrive at $- \nabla (A(\nabla \phi_\delta - \nabla  \phi)) = \nabla (\delta B(\nabla \phi_\delta + e))$, and so 
\begin{multline}\label{Gderivest}
 \frac{\phi_\delta(x) - \phi(x)}{\delta} = \int_{\T} -\nabla_y G(A;x,y) \cdot B(y) (\nabla  \phi_\delta(y) + e) \dy                                                                                             \\ = - \int_{\T} \nabla_y G(A;x,y) \cdot B(y) (\nabla \phi(y) + e) \dy + \int_{\T} \nabla_y G(A;x,y) \cdot B(y) (\nabla \phi(y) - \nabla \phi_\delta(y)) \dy.
\end{multline}
Since $B$ is smooth, using Schauder estimates we have that $\aa{ \nabla \phi - \nabla \phi_\delta}_{L^\infty(\T)} \to 0$ as $\delta \to 0$. This together with the fact that $\nabla_y G(A;x,y) \in L^1(\T)$ (see, e.g., \cite[Theorem 2]{DMGreen}) implies that the second term on the right-hand side of \eqref{Gderivest} goes to $0$ as $\delta \to 0$,
%
and \eqref{dphiA} follows. By Schauder estimates $\nabla \phi(\cdot) + y \in L^\infty(\T;Y^d)$, and so estimate $\a{\nabla_y G(x,y)} \le C\a{x-y}^{1-d}$ (see, e.g., \cite[Theorem 2]{DMGreen}), together with H\"older's inequality yields $\frac{\partial \phi(x)}{\partial A(\cdot)} \in L^1(\T;\mathcal{L}(Y^d,Y^d))$. 

\smallskip\noindent{\bf Step 3:}
It remains to prove \eqref{eq100}. We write
\begin{multline}\label{eq101}
\en{ \left( \int_{\T} \left( \int_{B_1(z)} \left| \frac{\partial (u(x) - u(0))}{\partial A(y)}\right| \dy \right)^2 \dz \right)^p } 
\\ 
\begin{aligned} 
 \le & \ C^{p-1} \en{ \left( \int_{B_4(0)} \left| \frac{\partial (u(x) - u(0))}{\partial A(y)}\right| \dy \right)^{2p} } 
\\
+ & \ C^{p-1} \en{ \left( \int_{|z| > 3} \left( \int_{B_1(z)} \left| \frac{\partial (u(x) - u(0))}{\partial A(y)}\right| \dy \right)^2 \dz \right)^p } 
\end{aligned}
\end{multline}
and estimate two terms on the right-hand side separately. 

For the first term, w.l.o.g. we assume that $x = t e_1$ for some $t \in [0,1)$. We define a curve $\gamma$, which consists of three line segments: $0\leftrightarrow -t e_d$, $-t e_d \leftrightarrow -t e_d + x$, and $-t e_d + x \leftrightarrow x$. Since $\frac{\partial (u(x) - u(0))}{\partial A(y)} = - (\nabla_y G(x,y) - \nabla_y G(0,y)) \nabla u(y)$ (proved in the previous step), we have (using notation $B_4^+(0) = \{ p \in B_4(0) : p \cdot e_d > 0 \}$):
\begin{multline}\nonumber
 \int_{B_4^+(0)} \left| \frac{\partial (u(x) - u(0))}{\partial A(y)} \right| \dy 
\le \int_{B_4^+(0)} \left| \nabla_y G(x,y) - \nabla_y G(0,y) \right| \left| \nabla u(y) \right| \dy 
\\
 \begin{aligned}
 &\le \int_\gamma \int_{B_4^+(0)} \left| \gg G(x',y)\right| \left| \nabla u(y)\right| \dy \ud{x'}
\\ &\!\!\!\!\!\overset{\textrm{H\"older}}{\le} \int_\gamma \left( \int_{B_4^+(0)} |y-x'|^{d+\frac{1}{2}} |\gg G(x',y)|^2 \dy \right)^\oh \left( \int_{B_4^+(0)} |y-x'|^{-d-\frac{1}{2}} |\nabla u(y)|^2 \dy \right)^\oh \ud{x'}
\\ &\!\!\overset{\eqref{eq1.2}}{\lesssim} \left( \int_\gamma  \int_{B_4^+(0)} |y-x'|^{-d-\frac{1}{2}} |\nabla u(y)|^2 \dy \ud{x'} \right)^\oh = \left( \int_{B_4^+(0)} \left( \int_\gamma |y-x'|^{-d-\frac{1}{2}} \ud{x'} \right) |\nabla u(y)|^2 \dy \right)^\oh,
 \end{aligned}
\end{multline}
where in the middle step we smuggled in weights $\a{y-x'}^{d+\frac{1}{2}}$ and $\a{y-x'}^{-d-\frac{1}{2}}$. 

By a simple scaling argument we observe that $\int_\gamma |y-x'|^{-d-\frac{1}{2}} \ud{x'} \lesssim \max(|y|,|x-y|)^{-d+\frac{1}{2}}$, and so 
\begin{equation}\nonumber
\begin{aligned}
  \int_{B_4^+(0)} &\left| \frac{\partial (u(x) - u(0))}{\partial A(y)} \right| \dy \lesssim \left( \int_{B_4^+(0)} \max(|y|,|x-y|)^{-d+\frac{1}{2}}|\nabla u(y)|^2 \dy \right)^\oh 
\\ &\!\!\!\!\overset{\textrm{H\"older}}{\lesssim} \left( \int_{B_4^+(0)} \max(|y|,|x-y|)^{-d+\frac{1}{4}} \dy \right)^{\frac{d-\frac{1}{2}}{2d-\frac{1}{2}}} \left( \int_{B_4^+(0)} |\nabla u(y)|^{8d-2} \dy \right)^{\frac{1}{8d-2}} 
\\ &\lesssim \left( \int_{B_4^+(0)} |\nabla u(y)|^{2p} \dy \right)^\frac{1}{2p},
 \end{aligned}
\end{equation}
for $2p = 8d-2$. By symmetry we have the same estimate with $B_4^+(0)$ on the left-hand side replaced by its complement in $B_4(0)$, and so 
\begin{equation}\nonumber
  \int_{B_4(0)} \left| \frac{\partial (u(x) - u(0))}{\partial A(y)} \right| \dy \lesssim  \left( \int_{B_4(0)} |\nabla u(y)|^{2p} \dy \right)^\frac{1}{2p}.
\end{equation}
By Jensen's inequality the previous estimate holds for all $2p \ge 8d-2$. Using 
local smoothness of $A$ and by considering slightly larger ball we can get the estimate for all $1 \le p <  \infty$:
\begin{equation}\nonumber
 \int_{B_4(0)} \left| \frac{\partial (u(x) - u(0))}{\partial A(y)} \right| \dy \lesssim  \left( \int_{B_5(0)} |\nabla u(y)|^{2p} \dy \right)^\frac{1}{2p}.
\end{equation}
Then by stationarity
\begin{multline}\nonumber
 \en{ \left( \int_{B_4(0)} \left| \frac{\partial (u(x) - u(0))}{\partial A(y)}\right| \dy \right)^{2p} }^{\frac{1}{2p}} \lesssim \en{ \left( \int_{B_5(0)} |\nabla u(y)|^{2p} \dy \right) }^{\frac{1}{2p}} 
\\ = \p{ \int_{B_5(0)} \en{ |\nabla u(y)|^{2p} } \dy }^{\frac{1}{2p}} = \p{ \int_{B_5(0)} \en{ |\nabla u(0)|^{2p} } \dy }^{\frac{1}{2p}}\lesssim \en{ |\nabla u(0)|^{2p} }^{\frac{1}{2p}}.
\end{multline}

\medskip
It remains to estimate the latter term on the right-hand side \eqref{eq101}. We again use $\frac{\partial (u(x) - u(0))}{\partial A(y)} = - (\nabla_y G(x,y) - \nabla_y G(0,y)) \nabla u(y)$ to write
\begin{multline}\nonumber
  \left( \int_{|z| \ge 3} \left( \int_{B_1(z)} \left| \frac{\partial (u(x) - u(0))}{\partial A(y)} \right| \dy \right)^2 \dz \right)^p 
\\ \begin{aligned}
&\le \left( \int_{|z| \ge 3} \left( \int_{B_1(z)} \int_0^1 \left| \gg G(tx,y) \right| |\nabla u(y)| \ud{t} \dy \right)^2 \dz \right)^p 
\\ &\le \int_0^1 \left( \int_{|z| \ge 3} \left( \int_{B_1(z)} \left| \gg G(tx,y) \right| |\nabla u(y)| \dy \right)^2 \dz \right)^p \ud{t}
\\ &\le C^p \int_0^1 \left( \int_{|z| \ge 3} \int_{B_1(z)} \left| \gg G(tx,y) \right|^2 |\nabla u(y)|^2 \dy \dz \right)^p \ud{t} 
\\ &\le C^p \int_0^1 \left( \int_{|y| \ge 2} \left| \gg G(tx,y) \right|^2 |\nabla u(y)|^2 \dy \right)^p \ud{t}
\\ &\le C^p \int_0^1 \left( \int_{|y| \ge 2} \left| \gg G(tx,y) \right|^2 \dy \right)^{p-1} \left( \int_{|y| \ge 2} \left| \gg G(tx,y) \right|^2 |\nabla u(y)|^{2p} \dy \right) \ud t 
\\ &\!\!\overset{\eqref{eq1.1+}}{\le} C^p \int_0^1 \int_{|y| \ge 2} \left| \gg G(tx,y) \right|^2 |\nabla u(y)|^{2p} \dy \ud t.
\end{aligned}
\end{multline}
Hence
\begin{multline}\nonumber
 \en{\left( \int_{|z| \ge 3} \left( \int_{B_1(z)} \left| \frac{\partial (u(x) - u(0))}{\partial A(y)} \right| \dy \right)^2 \dz \right)^p} 
\\ 
\begin{aligned}
&\le C^p \int_0^1 \int_{|y|\ge 2} \en{ \left| \gg G(tx,y) \right|^2 |\nabla u(y)|^{2p} } \dy \ud{t} 
\\ &\overset{\textrm{stationarity}}{=} C^p \int_0^1 \int_{|y|\ge 2} \en{ \left| \gg G(tx-y,0) \right|^2 |\nabla u(0)|^{2p} } \dy \ud{t} 
\\ &\le C^p \en{ \int_0^1 \int_{|y|\ge 1} \left| \gg G(0,y) \right|^2 \dy \ud{t}\ |\nabla u(0)|^{2p} } 
\\ &\!\!\overset{\eqref{eq1.1+}}{\le} C^p \en {|\nabla u(0)|^{2p}},
\end{aligned}
\end{multline}
which concludes the proof of Lemma~\ref{lm2}.
\end{proof}

\begin{proof}[Proof of Theorem \ref{thm1}]
\mbox{}\\{\bf Step 1:}
 The proof of \eqref{1moment} uses Lemma~\ref{lm2} and resembles Step 4 in the proof of Theorem \ref{thm3}. For $q \ge 2$ we have
\begin{equation}\nonumber
\begin{aligned}
 q \frac{\ud}{\ud{q}} \en{ \zeta^q } &= \en{ \zeta^q \ln \zeta^q } \overset{\eqref{1LSI}\textrm{ for } \zeta^\frac{q}{2}}{\le} \en{\zeta^q} \ln \en{ \zeta^q} + \frac{C}{\rho} \en{ \int_{\T} \left( \int_{B_1(z)} \left| \frac{\partial ( \zeta^{q/2} )}{\partial A(y)}\right| \dy \right)^2 \dz }
\\ &= \en{\zeta^q} \ln \en{ \zeta^q} + \frac{C}{\rho}q^2 \en{ \zeta^{q-2} \int_{\T} \left( \int_{B_1(z)} \left| \frac{\partial \zeta}{\partial A(y)}\right| \dy \right)^2 \dz } 
\\ &\!\!\!\!\overset{\textrm{H\"older}}\le \en{\zeta^q} \ln \en{ \zeta^q} + \frac{C}{\rho}q^2 \en{ \zeta^{q} }^{\frac{q-2}{q}} \en{ \p{ \int_{\T} \left( \int_{B_1(z)} \left| \frac{\partial \zeta}{\partial A(y)}\right| \dy \right)^2 \dz}^\frac{q}{2} }^\frac{2}{q}
\\ &\!\!\overset{\eqref{LSIest}}{\le} \en{\zeta^q} \ln \en{ \zeta^q} + \frac{C}{\rho}q^2 \en{ \zeta^{q} }.
\end{aligned}
\end{equation}
Denoting $F(q) := \en{ \zeta^q }$, similarly as before in Step 4 of the proof of Theorem 3 we get $F(q)^\frac{1}{q} \le \exp(\frac{C}{\rho}(q-2))F(2)^{\frac{1}{2}}$. 
Using energy estimates and stationarity of $\en{\cdot}$ we show $F(2) \lesssim 1$, which then implies \eqref{1moment}.
\medskip\mbox{}\\{\bf Step 2:}
To derive the error estimate \eqref{1error} from the moment bound \eqref{1moment}, we use that the Logarithmic Sobolev inequality implies the Spectral Gap inequality (see, e.g., \cite{ledouxbook}), i.e., that assuming \eqref{1LSI} we get 
\begin{equation}\label{T.10}
\en{ \p{ \zeta-\en{\zeta} }^2} \le \frac{1}{\rho} \en{ \int_{\T} \p{ \int_{B_1(z)} \a{ \frac{\partial \zeta}{\partial A(y)} } \dy }^2 \dz }
\end{equation}
for $\zeta\colon\Omega\rightarrow\mathbb{R}$ for which the right-hand side makes sense. 
\medskip\mbox{}\\{\bf Step 3}. Deterministic estimate of the vertical derivative. We claim
\begin{equation}\label{1ahom}
\a{ \frac{\partial \p{e_0 A_{hom}(A) e_1} }{\partial A(y)} } = L^{-d} \a{ \nabla \phi_0'(y) + e_0} \a{ \nabla \phi_1(y) + e_1},
\end{equation}
where $\phi'_0$ denotes corrector in direction $e_0$ for coefficient field $A^t$, adjoint of $A$. To show \eqref{1ahom}, consider two arbitrary coefficient fields $A,\bar A:=A+\delta B\in\Omega$ for $B$ smooth. We write for abbreviation $\phi_i(x)=\phi(A;x)$, $\bar \phi_i(x)=\phi(\bar A;x)$, $\phi_i'(x)=\phi(A^t;x)$, and $\bar \phi_i'(x)=\phi(\bar A^t;x)$ for correctors in directions $e_i$, $i=0,1$. By definition of the homogenized coefficient $A_{hom}$ 
\begin{eqnarray}\label{var}
\lefteqn{L^d(e_0 \cdot A_{hom}(\bar A)e_1 -e_0 \cdot A_{hom}(A)e_1)}\\ \nonumber
&\stackrel{\eqref{T.2}}{=}&
\int_{\T}e_0\cdot \bar A(\nabla\bar \phi_1+e_1)
-\int_{\T}e_0\cdot A(\nabla\phi_1+e_1)\\ \nonumber
&\overset{\eqref{T.3}\textrm{ for }A,\bar A}=&
\int_{\T}(\nabla\bar\phi_0'+e_0)\cdot \bar A(\nabla\bar\phi_1+e_1)
-\int_{\T}(\nabla\phi_0'+e_0)\cdot A(\nabla\phi_1+e_1)\\ \nonumber
&=&
\int_{\T}(\nabla\bar\phi_0'-\nabla\phi_0')\cdot \bar A(\nabla\bar\phi_1+e_1)
+\int_{\T}(\nabla\phi_0'+e_0)\cdot \bar A(\nabla\bar\phi_1+e_1)\\ \nonumber
&&-\int_{\T}(\nabla\phi_0' + e_0)\cdot A(\nabla \phi_1 - \nabla \bar \phi_1) - \int_{\T}(\nabla\phi_0'+e_0)\cdot A(\nabla\bar\phi_1+e_1)
\\ \nonumber &\overset{\eqref{T.3}\textrm{ for }\bar A,1A^t}=&\int_{\T}(\nabla\phi_0'+e_0)\cdot (\bar A-A)(\nabla\bar\phi_1+e_1)
\\ \nonumber &=&\int_{\T}(\nabla\phi_0'+e_0)\cdot (\bar A-A)(\nabla\phi_1+e_1) +
\int_{\T}(\nabla\phi_0'+e_0)\cdot (\bar A-A)(\nabla\bar\phi_1-\nabla\phi_1).
\end{eqnarray}

\noindent
Recalling that $\bar A = A+\delta B$, to show \eqref{1ahom} it is enough to show that the last term above divided by $\delta$ converges to $0$ as $\delta \to 0$ by showing that 
\begin{equation}\label{deriv1}
 \int_{\T}(\nabla\phi_0'+e_0)\cdot B (\nabla\bar\phi_1-\nabla\phi_1) \le \aa{B}_{L^\infty(\T)} \p{ \int_{\T} \a{ \nabla \phi_0' + e_0 }^2 }^{\frac{1}{2}} \p{ \int_{\T} \a{ \nabla \bar\phi_1 - \nabla \phi_1 }^2 }^{\frac{1}{2}} \to 0
\end{equation}
as $\delta \to 0$. 

Using \eqref{T.3} for $A$ and $A^t$ gives $\int_{\T} \a{ \nabla\phi'_0 + e_0}^2 \lesssim L^d$ and $\int_{\T} \a{ \nabla\phi_1 + e_1}^2 \lesssim L^d$. Since $-\nabla\cdot \bar A(\nabla \bar \phi_1 - \nabla \phi_1) = \delta \nabla \cdot B(\nabla \phi_1+e_1)$, we get the following  estimate $\int_{\T} \a{ \nabla \phi_1 - \nabla \phi_1 }^2 \lesssim \delta^2 \aa{B}^2_{L^\infty(\T)} \int_{\T} \a{ \nabla\phi_1 + e_1}^2 \lesssim L^d \delta^2 \aa{B}^2_{L^\infty(\T)}$, from where \eqref{deriv1} immediately follows. 
\medskip
\mbox{}\\{\bf Step 4}. Conclusion. In view of \eqref{T.10} and \eqref{1ahom} we have
\begin{align*}
\en{ \left(e_0\cdot A_{hom}e_1-\en{ e_0\cdot A_{hom}e_1}\right)^2} \le& \frac{1}{\rho}
\en{ \int_{\T} \p{ L^{-d} \int_{B_1(z)} \a{ \nabla \phi_0'(y) + e_0 } \a{ \nabla \phi_1(y) + e_1} \dy }^2 \dz } 
\\ 
=& \frac{L^{-2d}}{\rho} \int_{\T} \en{ \int_{B_1(z)} \a{ \nabla \phi_0'(y) + e_0 }^4 +  \a{ \nabla \phi_1(y) + e_1}^4 \dy } \dz 
\\ 
\overset{\eqref{stat}}\lesssim & \frac{L^{-2d}}{\rho} L^d \p{ \en{ \a{ \nabla \phi_0' + e_0 }^4 } + \en{ \a{ \nabla \phi_1 + e_1}^4 } }
\\ 
\overset{\eqref{1moment}}\lesssim& \frac{L^{-d}}{\rho},
\end{align*}
where to estimate $\en{ \a{ \nabla \phi'_0 + e_0 }^4 }$ in the last step using \eqref{1moment}, we used the fact that \eqref{1moment} holds also for the ensemble obtained by the push-forward of $A \mapsto A^t$. 
\end{proof}

\section{Proof of Theorem \ref{thm2}}\label{sec:pf2}

\begin{proof}[Proof of Theorem \ref{thm2}:]

\mbox{}\\ {\bf Step 1}. Stationarity: For any center $z \in \T$, radius $R > 0$ and exponent $1 \le p < \infty$ we have the identity
\begin{equation}\label{p.21}
\en{ \left( \int_{B_R(z)}|\nabla\phi+e|^2\right)^p} =
\en{ \left( \int_{B_R}|\nabla\phi+e|^2\right)^p}.
\end{equation}
Indeed, the stationarity (\ref{L2.11}) of $\phi$ also yields stationary of $\nabla\phi$, that is,
\begin{equation}\nonumber
\nabla\phi(A;x+z)=
\nabla\phi(A(\cdot+z);x)
\end{equation}
and thus
\begin{equation}\nonumber
\int_{B_R(z)}|\nabla\phi(A;x')+e|^2 \dx'
=\int_{B_R}|\nabla\phi(A(\cdot+z);x)+e|^2 \dx.
\end{equation}
By stationarity of $\en{ \cdot }$, cf.\ (\ref{stat}), applied to $\zeta(A)=(\int_{B_R(z)}|\nabla\phi(A;x)+e|^2 \dx)^p$, we get
\eqref{p.21}.
%

\medskip
\noindent{\bf Step 2}. For any $R \ge 1$ and any $1 \le p < \infty$ we have
\begin{equation}\label{p.14}
\en{ \left(\int_{B_{2R}}|\nabla\phi+e|^2\right)^p}
\le C(d)^p \en{ \left(\int_{B_R}|\nabla\phi+e|^2\right)^p}.
\end{equation}
%
Indeed, there exist points $z_1,\cdots,z_N$ on the torus such that
$B_{2R}\subset \bigcup_{n=1}^NB_R(z_n)$, where $N$ depends only on $d$. Thus we have
\begin{equation}\nonumber
\int_{B_{2R}}|\nabla\phi+e|^2\le\sum_{n=1}^N\int_{B_R(z_n)}|\nabla\phi+e|^2.
\end{equation}
We take the $p$-th power and apply H\"older's inequality 
\begin{equation}\nonumber
\left(\int_{B_{2R}}|\nabla\phi+e|^2\right)^p\le N^{p-1} \sum_{n=1}^N \left(\int_{B_R(z_n)}|\nabla\phi+e|^2\right)^p;
\end{equation}
taking the expectation then yields
\begin{equation}\nonumber
\en{\left(\int_{B_{2R}}|\nabla\phi+e|^2\right)^p} \le N^p \max_{n=1,\cdots,N} \en{\left(\int_{B_R(z_n)}|\nabla\phi+e|^2\right)^p}.
\end{equation}
By stationarity in form of \eqref{p.21} from Step 1, this together with $N=C(d)$  yields \eqref{p.14}.

\mbox{}\\{\bf Step 3.} Caccioppoli inequality. We claim that for any $R > 0$ and any $\bar u \in \R$
\begin{equation}\label{cac}
 \int_{B_R} \a{\nabla u}^2 \le \frac{C}{\lambda R^2} \int_{B_{2R} \setminus B_R} \a{u-\bar u}^2
\end{equation}
provided $\div A\nabla u = 0$ in $B_{2R}$. To prove \eqref{cac}, we test $\div A\nabla (u-\bar u) = 0$ with $\eta^2 (u-\bar u)$, where $\eta(x) := \min (1, \max ( 0, 2 - \a{x}/R))$ (i.e., $\eta$ is a cut-off function for $B_R$ in $B_{2R}$) to get
\begin{multline}\nonumber
\lambda \int_{\T} \a{\nabla \p{ (u-\bar u)\eta} }^2
\overset{\eqref{ellipticity}}\le \int_{\T}  \nabla ((u-\bar u) \eta)  \cdot A\nabla ((u-\bar u) \eta) 
\\ \le \int_{\T} 2 \a{ u-\bar u }\a{\nabla \eta} \a{\nabla (\eta (u-\bar u))} + \p{ u-\bar u}^2 \a{ \nabla \eta }^2. 
\end{multline}
where in the last inequality we used that $\a{A} \le 1$. Application of Young's inequality together with definition of $\eta$ then gives \eqref{cac}. 

\mbox{}\\{\bf Step 4.} Hole-filling argument (see for instance \cite[p. 81]{giaquinta}). 
We claim that there exists $\bar \alpha = \bar \alpha(d,\lambda) > 0$ such that for any $0 < \rho \le r$ we have 
\begin{equation}\label{e72}
 \int_{B_\rho} \a{ \nabla u }^2 \lesssim \p{ \frac{\rho}{r} }^{\bar \alpha} \int_{B_r} \a{ \nabla u }^2 \textrm{ provided } - \nabla\cdot A \nabla u = 0 \textrm{ in } B_r.
\end{equation}
Since we do not need a sharp estimate, we will prove \eqref{e72} only for $r = 2^n\rho$, $n \in \mathbb N$. Using \eqref{cac} from Step 3 and Poincar\' e inequality we get for any $0 < \sigma \le \frac{r}{2}$
\begin{equation}\nonumber
 \int_{B_\sigma} |\nabla u|^2 \overset{\eqref{cac}}{\le} \frac{C}{\lambda \sigma^2} \int_{B_{2\sigma}\setminus B_\sigma} \a{u - \bar u}^2 \le \frac{C(d)}{\lambda} \int_{B_{2\sigma}\setminus B_\sigma} \a{\nabla u}^2,
\end{equation}
where $\bar u$ is the average value of $u$ in $B_{2\sigma}\setminus B_\sigma$. 
Denoting $a_n := \int_{B_{2^n\rho}} \a{\nabla u}^2$, this implies $a_n \le \frac{C(d)}{\lambda} ( a_{n+1} - a_n)$. By moving $a_n$ to the left-hand side we get $a_n \le \theta a_{n+1}$ with $\theta = C(d)/(C(d)+\lambda) < 1$. We write $\theta = 2^{-\bar \alpha}$ (for some $\bar \alpha > 0$) and iterate the previous estimate to get $\int_{B_\rho} \a{\nabla u}^2 \le \theta^n \int_{B_{2^n\rho}} \a{\nabla u}^2 = (2^n)^{-\bar \alpha} \int_{B_{2^n\rho}} \a{\nabla u}^2$. The proof of this step is complete.

\mbox{}\\{\bf Step 5.} There exists $\alpha = \alpha(d,\lambda)$ such that for any $\rho \ge 1$
\begin{equation}\label{e74}
 \int_{B_\rho} \a{ \nabla u }^2 \lesssim \int_{\T} \p{ \frac{\a{x}}{\rho}+1 }^{-\alpha} \a{g}^2 
\textrm{ provided } - \nabla\cdot A \nabla u = \nabla \cdot g \textrm{ in } \T.
\end{equation}
To show \eqref{e74} we decompose $g = g_0 + \sum_{n=1}^\infty g_n$ with $g_0 = g\chi_{B_\rho}$ and $g_n = g \chi_{B_{2^n\rho}\setminus B_{2^{n-1}\rho}}$. By $u_n$ we denote the unique solution (with zero average over $\T$) of $-\nabla\cdot A\nabla u_n = \nabla \cdot g_n$. Since $\int_{\T} A\nabla u_n \nabla u_n = \int_{\T} - g_n \nabla u_n$, by \eqref{ellipticity} and H\"older's inequality we get $\int_{\T} \a{\nabla u_n}^2 \le \lambda^{-2} \int_{\T} \a{g_n}^2$. Since $-\div A \nabla u = 0$ in $B_{2^{n-1}\rho}$, we get for all $n\ge1$ 
\begin{equation}\nonumber
 \int_{B_\rho} \a{ \nabla u_n }^2 \overset{\eqref{e72}}\lesssim 2^{-(n-1)\bar \alpha} \int_{B_{2^{n-1}\rho}} \a{ \nabla u_n }^2 \lesssim 2^{-n\bar\alpha}\int_{\T} \a{g_n}^2.
\end{equation}
For $n=0$ we already obtained such estimate. Since $\nabla u = \sum_{n=0}^\infty \nabla u_n$, by triangle inequality in $L^2(B_\rho)$ we obtain
\begin{equation}\label{eq75}
 \begin{split}
 \p{ \int_{B_\rho} \a{\nabla u}^2 }^\frac{1}{2} &\le \sum_{n=0}^\infty \p{ \int_{B_\rho} \a{\nabla u_n}^2 }^\frac{1}{2} \lesssim \sum_{n=0}^\infty \p{ 2^{-n\bar\alpha} \int_{\T} \a{g_n}^2 }^\frac{1}{2} 
\\ 
&\overset{\textrm{H\"older}}{\lesssim} \p{ \sum_{n=0}^\infty  2^{-n\ep} }^\frac{1}{2} \p{ \sum_{n=0}^\infty 2^{-n\alpha} \int_{\T} \a{g_n}^2 }^\frac{1}{2}, 
\end{split}
\end{equation}
where $\ep > 0$ and $\alpha > 0$ s.t. $\bar\alpha = \alpha + \ep$. Since $g_n$ is supported in $B_{2^n\rho}\setminus B_{2^{n-1}\rho}$, we see that $2^{-n\alpha}\a{g_n(x)}^2 \lesssim \p{\frac{|x|}{\rho}+1}^{-\alpha}\a{g_n(x)}^2$. Hence it follows from \eqref{eq75} (using that $\ep > 0$, so that the last but one sum in \eqref{eq75} is summable)
\begin{equation}\nonumber
 \int_{B_\rho} \a{\nabla u}^2 \lesssim \sum_{n=0}^\infty 2^{-n\alpha} \int_{\T} \a{g_n}^2 \lesssim \sum_{n=0}^\infty \int_{\T} \p{\frac{|x|}{\rho}+1}^{-\alpha} \a{ g_n}^2 = \int_{\T} \p{\frac{|x|}{\rho}+1}^{-\alpha} \a{g}^2.
\end{equation}

{
\mbox{}\\{\bf Step 6.} 
For any $R \ge 1$ and $F$, linear (deterministic) functional on $L^2(B_{2R})$ that satisfies 
\begin{equation}\label{66}
  \a{ F(\nabla f) }^2 \lesssim \int_{B_{2R}} \a{ \nabla f }^2, \quad \forall f \in W^{1,2}(B_{2R})
\end{equation}
we have for $l=1,\ldots,n$ 
\begin{equation}\label{eq77}
 \p{ \sum_{z \in \ZL} \p{ \osc_{B_{\sqrt{d}}(z)} F(\nabla \phi_l + e_l)}^q }^\frac{2}{q} \lesssim \sup_{z \in \ZL} \p{ \frac{|z|}{R} + 1}^{-\alpha_2} \int_{B_{\sqrt{d}}(z)} \a{\nabla \phi + e}^2,
\end{equation}
where $\alpha_2>0$ and $1 < q < 2$ depend on $d$ and $\lambda$. 

Here comes the argument: Given a point on the integer lattice $z\in\mathbb{Z}^d \cap \T$, 
we denote by $A_z$ an arbitrary coefficient field that agrees with $A$ outside of $B_{\sqrt{d}}(z)$. We note that the function $\phi(A_z;\cdot)-\phi(A;\cdot)$ 
satisfies
\begin{equation}\label{L2.3}
-\nabla\cdot A\nabla(\phi(A_z;\cdot)-\phi(A;\cdot))=\nabla\cdot(A_z-A)(\nabla\phi(A_z;\cdot)+e).
\end{equation}
To estimate the $l^q(\ZL)$-norm on the left-hand side of \eqref{eq77} we will consider a discrete field $\{\omega_z\}_{z\in\mathbb{Z}^d\cap\T}$ and use definition for the $l^q$ norm by duality. Given $\{\omega_z\}_{z\in\mathbb{Z}^d\cap\T}$, we consider the 
function $u$ and the vector field $g$ defined through
\begin{eqnarray*}\nonumber
u(x)&:=&\sum_{z\in\mathbb{Z}^d\cap\T}\omega_z(\phi(A_z;x)-\phi(A;x)),\\
g(x)&:=&\sum_{z\in\mathbb{Z}^d\cap\T}\omega_z(A_z(x)-A(x))(\nabla\phi(A_z;x)+e),
\end{eqnarray*}
and note that (\ref{L2.3}) translates into $-\nabla\cdot A\nabla u=\nabla\cdot g$. We combine assumption \eqref{66} with \eqref{e74} from Step 6 to obtain
\begin{equation}\label{102}
 \a{ F(\nabla u_l) }^2 \overset{\eqref{66}}\lesssim \int_{B_{2R}} \a{\nabla u}^2 \overset{\eqref{e74}}\lesssim \int_{\T} \p{ \frac{|x|}{2R}+1 }^{-\alpha} |g|^2 \lesssim \int_{\T} \p{ \frac{|x|}{R}+1 }^{-\alpha} |g|^2. 
\end{equation}
By the linearity of $F$, the left-hand side can be written as
\begin{equation}\nonumber
\a{ F(\nabla u_l)}^2=\a{\sum_{z\in\ZL}\omega_z(F(\nabla\phi_l(A_z;\cdot)+e_l)-F(\nabla\phi_l(A;\cdot)+e_l))}^2.
\end{equation}
Since $\a{A_z-A}\le2$ is supported in $B_{\sqrt{d}}(z)$ and since balls $B_{\sqrt{d}}(z)$ have finite overlap (which depends only on $d$), the right-hand side of \eqref{102} is estimated by
\begin{align*}
\int_{\T} \p{ \frac{|x|}{R}+1}^{-\alpha}|g|^2 
&= \int_{\T} \p{ \frac{|x|}{R}+1}^{-\alpha}\a{ \sum_{z\in\mathbb{Z}^d\cap\T}\omega_z(A_z(x)-A(x))(\nabla\phi(A_z;x)+e) }^2 
\\ &\lesssim
\sum_{z\in\ZL} \p{ \omega_z^2\int_{B_{\sqrt{d}}(z)}\p{\frac{|x|}{R}+1}^{-\alpha}\a{\nabla\phi(A_z;x)+e}^2 }
\\ &\lesssim \sum_{z\in\ZL}\omega_z^2\p{\frac{|z|}{R}+1}^{-\alpha}\int_{B_{\sqrt{d}}(z)}|\nabla\phi(A;\cdot)+e|^2,
\end{align*}
where in the last step we used $R \ge 1$ and 
$$\int_{B_{\sqrt{d}}(z)}|\nabla\phi(A_z;\cdot)-\nabla\phi(A;\cdot)|^2\le\int_{\T}|\nabla\phi(A_z;\cdot)-\nabla\phi(A;\cdot)|^2
\lesssim \int_{B_{\sqrt{d}}(z)}|\nabla\phi(A;\cdot)+e|^2,$$ 
where the second inequality follows from testing $-\nabla\cdot A_z\nabla(\phi(A_z;\cdot)-\phi(A;\cdot))=\nabla\cdot(A_z-A)(\nabla\phi(A;\cdot)+e)$ with $\phi(A_z;\cdot) - \phi(A;\cdot)$, and using \eqref{ellipticity} and H\"older's inequality. 

Now we split $\alpha = \alpha_1 + \alpha_2$, $\alpha_1 > 0$, $\alpha_2 > 0$, and based on the previous estimate we get
\begin{multline}\nonumber
 \int_{\T} \p{ \frac{|x|}{R}+1}^{-\alpha}|g|^2 
\\
\begin{aligned}
& \lesssim \p{ \sum_{z\in\ZL}\omega_z^2 \p{\frac{|z|}{R}+1}^{-\alpha_1 } } \sup_{z \in \ZL} \p{\frac{|z|}{R}+1}^{-\alpha_2}\int_{B_{\sqrt{d}}(z)}|\nabla\phi(A;\cdot)+e|^2
\\ 
& \lesssim \p{ \sum_{z\in\ZL}\omega_z^{2p} }^\frac{1}{p}\!\!\! \p{ \sum_{z \in \ZL} \p{\frac{|z|}{R}+1}^{-\alpha_1p' } }^\frac{1}{p'}\!\!\!\!\!\!\! \sup_{z \in \ZL} \p{\frac{|z|}{R}+1}^{-\alpha_2}\!\!\!\int_{B_{\sqrt{d}}(z)}|\nabla\phi(A;\cdot)+e|^2,
\end{aligned}
\end{multline}
for any $p,p' \ge 1$, $\frac{1}{p} + \frac{1}{p'} = 1$. We choose $p' > \frac{d}{\alpha_1}$ so that the second sum is finite,  and obtain
\begin{align}\nonumber
 \int_{\T} \p{ \frac{|x|}{R}+1}^{-\alpha}|g|^2 
&\lesssim \p{ \sum_{z\in\ZL}\omega_z^{2p} }^\frac{1}{p} \sup_{z \in \ZL} \p{\frac{|z|}{R}+1}^{-\alpha_2}\int_{B_{\sqrt{d}}(z)}|\nabla\phi(A;\cdot)+e|^2.
\end{align}
Since $\{\omega_z\}_{z\in\ZL}$ was arbitrary, by duality
\begin{multline}\nonumber
\p{ \sum_{z\in\ZL}|F(\nabla\phi_l(A_z;\cdot)+e_l)-F(\nabla\phi_l(A;\cdot)+e_l))|^q }^\frac{2}{q}
\\
\lesssim \sup_{z \in \ZL} \p{\frac{|z|}{R}+1}^{-\alpha_2}\int_{B_{\sqrt{d}}(z)}|\nabla\phi(A;\cdot)+e|^2,
\end{multline}
where $\frac{1}{2p} + \frac{1}{q} = 1$. 
Recalling definition of $\osc_{B_{\sqrt{d}}(z)} F$, the above implies \eqref{eq77}.


\mbox{}\\{\bf Step 7.} Compactness. There exists $N=N(d,\lambda)$ such that for any radius $1 \le R < \infty$ there exist $F_1,\ldots,F_N$, linear (deterministic) functionals on $L^2(B_{2R})$, with the following properties:
\begin{itemize}
 \item They satisfy \eqref{66}, i.e., they are bounded in the sense $\a{ F_k(\nabla f) }^2 \lesssim \int_{B_{2R}} \a{ \nabla f }^2$ for all $f \in W^{1,2}(B_{2R})$ and $k = 1,\ldots,N$;
 \item Together, they are strong enough to guarantee 
 \begin{equation}\label{67}
  \en{ \p{ \int_{B_R} \a{ \nabla \phi + e}^2 }^p } \le C(d,n,\lambda)^p \max_{
\genfrac{}{}{0pt}{}{k=1,\ldots,N}{l=1,\ldots,n}} \en{ \a{ F_k(\nabla \phi_l + e_l)}^{2p}}.
 \end{equation} 
\end{itemize}
\noindent
Here comes the argument. For $k \in \mathbb{N}$, let $\hat v_k$ be the $k$-th Neumann eigenfunction in the ball $B_{2}$ with $\mu_k$ being the corresponding eigenvalue, i.e., $-\Delta \hat v_k = \mu_k \hat v_k \textrm{ in } B_{2}$, $\frac{\partial \hat v_k}{\partial \nu} = 0 \textrm{ on } \partial B_{2}$, $\|\hat v_k\|_{L^2(B_{2})} = 1$, and $0 = \mu_0 < \mu_{1} \le \ldots$. We claim that the choice
\begin{equation}\nonumber
 F_k(f) := R \int_{B_{2R}} f \frac{\nabla v_k}{\mu_k}, \qquad k \ge 1
\end{equation}
where $v_k(x) := R^{-d/2} \hat v_k(x/R)$ for $x \in B_{2R}$, satisfies both \eqref{66} and \eqref{67} provided $N$ is chosen large enough. 

Functions $\{ v_k \}_{k=0}^\infty$ form an orthonormal basis in $L^2(B_{2R})$ and satisfy $-\Delta v_k = \frac{\mu_k}{R^2} v_k$. Given $u \in L^2(\T)$, we write $u = \sum_{k=0}^\infty c_k v_k$, and get 
for any $N \in \mathbb{N}$ 
\begin{equation}\label{77}
 \int_{B_{2R}} \a{ u - \bar u }^2 = \sum_{k=1}^N c_k^2 + \sum_{k=N+1}^\infty c_k^2 \le \sum_{k=1}^N \a{ \int_{B_{2R}} (u-\bar u) v_k }^2 + \frac{R^2}{\mu_{N+1}} \int_{B_{2R}} \a{ \nabla u }^2,
\end{equation}
where $\bar u$ denotes the average value of $u$ over $B_{2R}$. We will use this inequality for $u_l(x) := \phi_l(x) + x \cdot e_l$. Taking the $p$-th power and ensemble average in the previous relation, and by using Young's inequality we get
\begin{multline}\label{68}
 \en{ \p{ \int_{B_{R}} \a{ \nabla \phi_l + e_l }^2 }^p } = \en{ \p{ \int_{B_{R}} \a{ \nabla u_l }^2 }^p } \overset{\eqref{cac}}
\le \en { \p {\frac{C}{\lambda R^2} \int_{B_{2R}} \a{ u_l - \bar u_l }^2 }^p } 
\\ \overset{\eqref{77}}\le C^p \lambda^{-p} R^{-2p} (2N)^{p-1} \sum_{k=1}^N \en{ \a{ \int_{B_{2R}} (u_l-\bar u_l) v_k }^{2p} } + \frac{C}{2} \p{ \frac{2}{\lambda \mu_{N+1}}}^p \en {\p{ \int_{B_{2R}} \a{ \nabla u_l }^2}^p }.
\end{multline}
Since $\int_{B_{2R}} \nabla f \frac{\nabla v_k}{\mu_k} = \frac{1}{R^2} \int_{B_{2R}} f v_k$ for any $f \in W^{1,2}(B_{2R})$, we see that $\int_{B_{2R}} (u_l-\bar u_l) v_k = R F_k(\nabla \phi_l + e_l)$. Therefore
\begin{multline}\label{961}
 \en{ \p{ \int_{B_{R}} \a{ \nabla \phi + e }^2 }^p } \le C^p \sum_{l=1}^n \en{ \p{ \int_{B_{R}} \a{ \nabla \phi_l + e_l }^2 }^p } 
\\ \overset{\eqref{68}}\le C^p (2N)^{p-1} \sum_{l=1}^n \sum_{k=1}^N \en{ \a{ F_k\p{\nabla \phi_l + e_l} }^{2p} } + C^p \mu_{N+1}^{-p} \sum_{l=1}^n \en {\p{ \int_{B_{2R}} \a{ \nabla \phi_l + e_l }^2}^p }.
\end{multline}
We have $\en {\p{ \int_{B_{2R}} \a{ \nabla \phi_l + e_l }^2}^p } \le \en {\p{ \int_{B_{2R}} \a{ \nabla \phi + e }^2}^p }$, and so the last term from \eqref{961} satisfies
\begin{multline}\label{fourier}
 C^p \mu_{N+1}^{-p} \sum_{l=1}^n \en {\p{ \int_{B_{2R}} \a{ \nabla \phi_l + e_l }^2}^p } 
\\
\le C^p \mu_{N+1}^{-p} n \en {\p{ \int_{B_{2R}} \a{ \nabla \phi + e }^2}^p } \overset{\eqref{p.14}}\le C^p \mu_{N+1}^{-p} n \en {\p{ \int_{B_{R}} \a{ \nabla \phi + e }^2}^p }.
\end{multline}
Since $\mu_N \to \infty$ as $N \to \infty$, for large enough $N = N(d,n,\lambda)$ the prefactor in front of the last term in \eqref{fourier} satisfies $C^p \mu_{N+1}^{-p} n \le 1/2$. Therefore, for such $N$ we can absorb the last term in \eqref{961} into left-hand side and obtain
\begin{equation}\nonumber
 \en{ \p{ \int_{B_{R}} \a{ \nabla \phi + e }^2 }^p } \le C^p \sum_{l=1}^n \sum_{k=1}^N \en{ \a{ F_k\p{\nabla \phi_l + e_l} }^{2p} } \le C^p \max_{\genfrac{}{}{0pt}{}{k=1,\ldots,N}{l=1,\ldots,n}} \en{ \a{ F_k(\nabla \phi_l + e_l)}^{2p}},
\end{equation}
which is exactly \eqref{67}.

It remains to show that all $F_k$ satisfy upper bound \eqref{66}. By definition of $F_k$ and $v_k$
\begin{equation}\nonumber
 \a{ F_k(\nabla f) } = \a{ R \int_{B_{2R}} \nabla f \frac{\nabla v_k}{\mu_k} } \le R \p{ \int_{B_{2R}} \a{ \nabla f }^2 }^\oh \p{ \int_{B_{2R}} \frac{\a{\nabla v_k}^2}{\mu_k^2} }^\oh = \mu_k^{-1/2} \p{ \int_{B_{2R}} \a{ \nabla f }^2 }^\oh,
\end{equation}
where we used that $\int_{B_{2R}} \nabla v_k \nabla v_k = \int_{B_{2R}} v_k (-\Delta v_k) = \frac{\mu_k}{R^2} \int_{B_{2R}} v_k^2 = \frac{\mu_k}{R^2}$. Since $\mu_k \gtrsim 1$, we get \eqref{66}.

\mbox{}\\{\bf Step 8.} Proof of \eqref{2moment}. Since all $F_k$ constructed in the previous step satisfy \eqref{66}, they also satisfy \eqref{eq77}. Then taking the p-th power and expectation in \eqref{eq77} gives
\begin{equation}\label{106}
\begin{aligned}
 \en{ \p{ \sum_{z \in \ZL} \p{ \osc_{B_{\sqrt{d}}(z)} F_k(\nabla \phi_l + e_l)}^q }^{\frac{2p}{q}} }
&\lesssim \en{ \sup_{z \in \ZL} \p{ \frac{|z|}{R} + 1}^{-p\alpha_2} \p{ \int_{B_{\sqrt{d}}(z)} \a{\nabla \phi + e}^2}^p }
\\ &\lesssim \sum_{z \in \ZL} \p{ \frac{|z|}{R} + 1}^{-p\alpha_2} \en{ \p{ \int_{B_{\sqrt{d}}(z)} \a{\nabla \phi + e}^2}^p }
\\ &= \sum_{z \in \ZL} \p{ \frac{|z|}{R} + 1}^{-p\alpha_2} \en{ \p{ \int_{B_{\sqrt{d}}} \a{\nabla \phi + e}^2}^p },
\end{aligned}
\end{equation}
where we have used stationarity (see \eqref{p.21} from Step 1) in the last equality. We choose $p$ large enough so that 
\begin{equation}\nonumber
 p\alpha_2 > d.
\end{equation}
Then for $R \ge 1$ we have 
\begin{equation}\nonumber
 \sum_{z \in \ZL} \p{ \frac{|z|}{R} + 1}^{-p\alpha_2} \lesssim R^d. 
\end{equation}
On the other hand, by \eqref{e72} in Step 4 we have for $R \ge \sqrt{d}$
\begin{equation}\nonumber
 \int_{B_{\sqrt{d}}} |\nabla \phi + e|^2 \lesssim R^{-\alpha} \int_{B_R} |\nabla \phi + e|^2.
\end{equation}
Therefore \eqref{106} turns into
\begin{equation}\nonumber
 \en{ \p{ \sum_{z \in \ZL} \p{ \osc_{B_{\sqrt{d}}(z)} F_k(\nabla \phi_l + e_l)}^q }^\frac{2p}{q} } \lesssim R^{d-\alpha_2 p} \en{ \p{ \int_{B_R} \a{\nabla \phi + e}^2}^p }.
\end{equation}
Now we want to use the following $L^p$ version of the Spectral Gap inequality 
\begin{equation}\label{SGp}
 \en{ \a{\zeta - \en{\zeta}}^{2p} } \le C(\rho,p) \en{ \p{ \sum_{z\in\ZL} \p{ \osc_{B_\d(z)} \zeta}^q }^{\frac{2p}{q}} }
\end{equation}
for random variables $\zeta \in L^{2p}$ for which the right-hand side makes sense (see \cite[Lemma 11]{GNO2} for a similar inequality). 
For convenience of the reader we show in the Appendix how \eqref{SG} implies \eqref{SGp}. Using \eqref{SGp} for $\zeta = F_k(\nabla \phi_l + e_l)$ we get
\begin{align*}
 \en{ \a{F_k(\nabla \phi_l + e_l)}^{2p}} 
&\le C(d,n,\lambda,\rho,p) \p{ \en{ F_k^2(\nabla \phi_l + e_l) }^p + R^{d-\alpha_2 p}\en{ \p{ \int_{B_R} \a{\nabla \phi + e}^2}^p } }
\\ &\overset{\eqref{66}}{\le} C(d,n,\lambda,\rho,p) \p{ \en{ \int_{B_{2R}} |\nabla \phi_l + e_l|^2 }^p + R^{d-\alpha_2 p}\en{ \p{ \int_{B_R} \a{\nabla \phi + e}^2}^p } }
\\ &\le C(d,n,\lambda,\rho,p) \p{ R^{dp} + R^{d-p\alpha_2}\en{ \p{ \int_{B_R} \a{\nabla \phi + e}^2}^p }}.
\end{align*}
Plugging this into \eqref{67} implies
\begin{multline}\nonumber
 \en{ \p{ \int_{B_R} \a{ \nabla \phi + e}^2 }^p } \le C(d,n,\lambda,\rho,p) \max_{\genfrac{}{}{0pt}{}{k=1,\ldots,N}{l=1,\ldots,n}} \en{ \a{ F_k(\nabla \phi_l + e_l)}^{2p}} 
\\
\le C(d,n,\lambda,\rho,p) \p{ R^{dp} + R^{d-p\alpha_2}\en{ \p{ \int_{B_R} \a{\nabla \phi + e}^2}^p } }.
\end{multline}
Since $d-p\alpha_2 < 0$, we can choose $R=R_0(d,n,\lambda,\rho,p)$ large enough in order to absorb the last term into left-hand side to conclude
\begin{equation}\nonumber
  \en{ \p{ \int_{B_R} \a{ \nabla \phi + e}^2 }^p } \le C(d,n,\lambda,\rho,p)
\end{equation}
for $p$ sufficiently large. 

\mbox{}\\{\bf Step 9.} Similarly as in the proof of \eqref{1error}, Spectral Gap estimate \eqref{SG} for $\zeta := e_0 \cdot A_{hom} e_1$ and \eqref{2moment} imply \eqref{2error}. 

Indeed, it follows from \eqref{var} (and discussion afterwards) that for $A \in \Omega$ one has
\begin{equation}\nonumber
 L^d \osc_{B_{\sqrt{d}}(z)} e_0 \cdot A_{hom}(A)e_1 \lesssim \int_{B_{\sqrt{d}}(z)} \a{ \nabla \phi'_0 + e_0 }^2 + \a{ \nabla \phi_1 + e_1 }^2.
\end{equation}
Here we used notation used in \eqref{var}. Hence, \eqref{SG} with $q=2$ used for $\zeta(A) := e_0 \cdot A_{hom}(A)e_1$ implies
\begin{align*}
\en{ \left(e_0\cdot A_{hom}e_1-\en{ e_0\cdot A_{hom}e_1}\right)^2} &\lesssim \en{ \sum_{z\in\ZL} \p{ \osc_{B_{\sqrt{d}}} e_0A_{hom}e_1 }^2 }
\\ &\lesssim L^{-2d}  \sum_{z\in\ZL} \en{ \p{ \int_{B_{\sqrt{d}}(z)} \a{ \nabla \phi'_0 + e_0}^2 + \a{\nabla \phi_1 + e_1}^2 }^2 }
\\ &\!\!\overset{\eqref{p.21}}\lesssim L^{-d} \en{ \p{ \int_{B_{\sqrt{d}}} \a{ \nabla \phi'_0 + e_0}^2 }^2 + \p{ \int_{B_{\sqrt{d}}} \a{\nabla \phi_1 + e_1}^2 }^2 } 
\\ &\!\!\overset{\eqref{2moment}}\le C(d,n,\lambda,\rho) L^{-d}.
\end{align*}

}\end{proof}

\section{Proof of the $L^p$-version of the Spectral Gap estimate}

The $L^p$ version of the Spectral Gap estimate is a consequence of the standard Spectral Gap estimate, and was used (with $q=2$) previously in several works with suboptimal dependence on $p$ \cite{GNO2,AMN}. Recently, a sharp version in terms of $p$ was needed in \cite{GNO4}.  Since the case $q \neq 2$ did not appear previously in the literature, for the convenience of the reader we present the proof. Since we do not aim to get optimal $p$-dependence of the constants, we follow a simpler approach presented in \cite{GNO2,AMN} (in order to get the optimal $p$-dependence, the proof in \cite{GNO4} requires some new ideas). 


\begin{proof}[Proof of \eqref{SGp}.] Given $\zeta \in L^{2p}(\Omega)$ with $\en{\zeta} = 0$ we want to prove
\begin{equation}\label{113}
 \en{ \a{\zeta - \en{\zeta}}^{2p} } \lesssim \en{ \p{ \sum_{z\in\ZL} \p{ \osc_{\sqrt{d}} \zeta}^q }^{\frac{2p}{q}} },
\end{equation}
provided 
\begin{equation}\label{114}
 \en{ \p{ \zeta - \en{\zeta} }^{2} } \le \frac{1}{\rho} \en{ \p{ \sum_{z\in\ZL} \p{ \osc_{\sqrt{d}} \zeta}^q }^{\frac{2}{q}} }
\end{equation}
holds for any $\zeta \in L^2(\Omega)$. Here $\lesssim$ stays for $\le$ up to a constant depending on $p,q,\rho$. 

\mbox{}\\ {\bf Step 1}. We claim 
\begin{equation}\label{115}
 \p{ \osc_{\sqrt{d}} f^p }^q \lesssim f^{q(p-1)} \p{ \osc_{\sqrt{d}} f }^q + \p{ \osc_{\sqrt{d}} f }^{pq}.
\end{equation}

Indeed, from the elementary real-variable estimate
\begin{equation}\nonumber
 \a{ \zeta^p - \widetilde{\zeta}^p } \lesssim \a{ \zeta } ^{p-1} \a{ \zeta - \widetilde{\zeta} } + \a{ \zeta - \widetilde{\zeta} }^p,
\end{equation}
we get by definition of $\osc$
\begin{equation}\nonumber
 \osc_{\sqrt{d}} \zeta^p  \lesssim \a{\zeta}^{(p-1)} \osc_{\sqrt{d}} \zeta + \p{ \osc_{\sqrt{d}} \zeta }^{p}.
\end{equation}
We take $q$-th power ($q \ge 1$) of the previous relation, and use Young's inequality to obtain \eqref{115}. 

\mbox{}\\{\bf Step 2.} Fix $\zeta \in L^{2p}(\Omega)$. W.l.o.g. we can assume $\en{\zeta} = 0$. Then using \eqref{114} for $\zeta^p$ we get 
\begin{eqnarray*}
 \en{ \a{\zeta}^{2p} } &\le& \en{ \a{\zeta}^p }^2 + \frac{1}{\rho} \en{ \p{ \sum_{z\in\ZL} \p{ \osc_{\sqrt{d}} \zeta^p }^q }^\frac{2}{q} } 
\\
&\overset{\eqref{115}}\lesssim& \en{ \a{\zeta}^p }^2 + \en{ \p{ \sum_{z\in\ZL} \a{\zeta}^{q(p-1)} \p{ \osc_{\sqrt{d}} \zeta}^q   + \p{ \osc_{\sqrt{d}} \zeta }^{pq} }^\frac{2}{q} }
\\ 
&\overset{\textrm{Young's}}\lesssim& \en{ \a{\zeta}^p }^2 + \en{ \a{\zeta}^{2(p-1)}\p{ \sum_{z\in\ZL} \p{ \osc_{\sqrt{d}} \zeta}^q }^\frac{2}{q}} + \en{ \p{ \sum_{z\in\ZL} \p{ \osc_{\sqrt{d}} \zeta }^{pq} }^\frac{2}{q} }
\end{eqnarray*}
We now estimate three terms on the right-hand side separately. For the last term we appeal to the discrete $l^p$-$l^1$ estimate. By H\"older's inequality with exponents $\p{\frac{p}{p-1},p}$ and Young's inequality, the middle term is estimated by $\frac{1}{4} \en{ \a{ \zeta}^{2p} } + C \en{ \p{ \sum_{z\in\ZL} \p{ \osc_{\sqrt{d}} \zeta}^q }^{\frac{2p}{q}} }$. 
If $p \le 2$, by Jensen's inequality the first term is bounded by $\en{ \a{ \zeta }^2 }^p$. In the case $p > 2$, we combine H\"older's and Young's inequality to get $\en{ \a{\zeta}^p }^2 \le C\en{ \a{\zeta}^2 }^p + \frac{1}{4} \en{ \a{ \zeta }^{2p} }$. Finally, since we assumed $\en {\zeta} = 0$, \eqref{114} and Young's inequality imply $\en{ \a{ \zeta }^2 }^p \lesssim \en{ \p{ \sum_{z\in\ZL} \p{ \osc_{\sqrt{d}} \zeta}^q }^{\frac{2p}{q}} }$. Summing these estimates together yields \eqref{113}, which concludes the proof of \eqref{SGp}. 
\end{proof}

\bibliographystyle{amsplain}
\bibliography{bella}

%
%
%
%
%

\end{document}